\documentclass[12 pt]{amsart}
\usepackage{amsmath, fullpage}
\usepackage{graphicx,tikz}
\usepackage{tikz-cd}
\usepackage{caption}
\usepackage{comment}
\usepackage{mathtools}
\usepackage{amsthm}
\usepackage{amsfonts}
\usepackage{hyperref}
\usepackage{amssymb}
\usepackage{bibentry}
\usepackage{enumerate}
\newtheorem{theorem}{Theorem}
\theoremstyle{definition}{\newtheorem{remark}[theorem]{Remark}}
\newtheorem{lemma}[theorem]{Lemma}
\newtheorem{conjecture}[theorem]{Conjecture}
\newtheorem{question}[theorem]{Question}
\newtheorem{example}[theorem]{Example}
\newtheorem{proposition}[theorem]{Proposition}
\theoremstyle{definition}{\newtheorem{definition}[theorem]{Definition}}
\newtheorem{corollary}[theorem]{Corollary}
\title{On the Packing Functions of Some Linear Sets of Lebesgue Measure Zero}
\author{Austin Anderson}
\author{Steven Damelin}
\address{Department of Applied Mathematics, Florida Polytechnic University, Lakeland, FL, 33805}
\email{aanderson@floridapoly.edu}
\address{zbMATH Open, Department of Mathematics, FIZ Karlsruhe - Leibniz Institute for Information Infrastructure, Berlin, Germany, Franklinstr. 11 10587}
\email{steve.damelin@gmail.com}
\begin{document}
\begin{abstract}
We use a characterization of Minkowski measurability to study the asymptotics of best packing on cut-out subsets of the real line with Minkowski dimension $d\in(0,1)$. Our main result is a proof that Minkowski measurability is a sufficient condition for the existence of best packing asymptotics on monotone rearrangements of these sets. For each such set, the main result provides an explicit constant of proportionality $p_d,$ depending only on the Minkowski dimension $d,$ that relates its packing limit and Minkowski content. We later use the Digamma function to study the limiting value of $p_d$ as $d\to 1^-.$
\par For sharpness, we use renewal theory to prove that the packing constant of the $(1/2,1/3)$ Cantor set is less than the product of its Minkowski content and $p_d$. 
\par We also show that the measurability hypothesis of the main theorem is necessary by demonstrating that a monotone rearrangement of the complementary intervals of the 1/3 Cantor set has Minkowski dimension $d=\log2/\log3\in(0,1),$ is not Minkowski measurable, and does not have convergent first-order packing asymptotics. 
\par The aforementioned characterization of Minkowski measurability further motivates the asymptotic study of an infinite multiple subset sum problem.
\end{abstract}
\maketitle
\section{Introduction}
\begin{definition}[Best packing] For a positive integer $p\geq 1,$ an infinite compact set $A\subset\mathbb{R}^p,$ and a finite point set $\omega_N=\{x_1,\ldots,x_N\}\subset A,$  define the \textit{separation of $\omega_N$} by 
$$\delta(\omega_N):=\min_{i\neq j}||x_i-x_j||_p,$$
where $||\cdot||_p$ denotes the Euclidean norm on $\mathbb{R}^p.$
Then the \textit{$N$-point best packing radius of $A$} is given by 
$$\delta(A,N):=\max_{\omega_N\subset A}\delta(\omega_N).$$
We are primarily interested in studying the generalized inverse of $\delta(A,N)$, the \textit{packing function of $A$}
$$N(A,\varepsilon):=\max\{N\geq 2: \delta(A,N)\geq \varepsilon\},\quad \text{ where } 0<\varepsilon\leq \text{diam}(A).$$
\end{definition}
\par Because $A$ is bounded, its size and geometry dictate the asymptotic growth of $N(A,\varepsilon)$ as $\varepsilon\to 0^+.$ In this manuscript, we study the asymptotic behavior of $N(A,\varepsilon)$ as $\varepsilon\to 0^+$ over certain subsets $A\subset\mathbb{R}$ that evade existing techniques. 
\begin{subsection}{Main Theorems}
For the remainder of the manuscript, let $m_p$ denote Lebesgue measure on $\mathbb{R}^p.$   
\par For each infinite compact set $A\subset\mathbb{R},$ there is a closed interval $I$ of minimal length such that $A\subset I.$ Because $I\setminus A$ is open and bounded in $\mathbb{R}$, $I\setminus A=\cup_{j=1}^{\infty}I_j$ for a collection of pairwise disjoint, bounded, open intervals $I_j$ of positive, finite length. Therefore, $$A=I\setminus(I\setminus A)=I\setminus \cup_{j=1}^{\infty}I_j,$$ as $A\subset I.$ In this manuscript, we are primarily interested in the asymptotic behavior of $N(A,\varepsilon)$ over certain compact sets of Lebesgue measure $0$ that are not covered by known techniques. Notice that if $m_1(A)=0,$ then $m_1(I)=\sum_{j=1}^{\infty}m_1(I_j),$ by the minimality of the interval $I.$ Hence, all compact subsets $A\subset\mathbb{R}$ with $m_1(A)=0$ take the following form.
\begin{definition}[Cut-out sets and their rearrangements]\label{def:cut}
Suppose $I\subset \mathbb{R}$ is a compact interval, and consider a countable collection of disjoint open intervals $I_j$ contained in $I,$ $$\bigcup_{j=1}^{\infty}I_j\subset I,$$ indexed so that the sequence $l_j=m_1(I_j)>0$ is non-increasing for $j\geq 1.$ Then assume  $$\sum_{j=1}^{\infty}l_j=m_1(I).$$ We refer to the sets $\Gamma$ obtained in this manner as \textit{cut-out sets}. 
\par If $\Gamma'=I\setminus \cup_{j=1}^{\infty}I_j'$ is any another cut-out set built by removing complementary intervals of the same lengths $|I_j'|=l_j$ from $I$ in a possibly different geometric order, then we say that $\Gamma'$ is a \textit{rearrangement} of $\Gamma.$ Denote the set of all rearrangements of $\Gamma$ by $\mathcal{R}\left(\Gamma\right).$ Any rearrangement of $\Gamma$ with the complementary intervals $I_j$ placed in left-to-right order of non-increasing lengths shall be referred to as a \textit{monotone rearrangement} of $\Gamma.$ We will often shorten this by simply saying that $\Gamma$ is \textit{monotone}.
\end{definition} 
Building on Borel's work \cite{Borel}, Beisicovtich and Taylor developed various notions of dimension for cut-out sets and their rearrangements in \cite{Besicovitch_Taylor_1954}. These have led to many subsequent studies (e.g. \cite{p-cant},\cite{Cabrelli_Mendivil_Molter_Shonkwiler_2004},\cite{Xiong_Wu_2009},\cite{Hare_Mendivil_Zuberman_2013b}).
The Main Theorem of this manuscript is a pair of conditions that suffice to establish the existence of first-order asymptotics for $N(\Gamma,\varepsilon)$ on cut-out sets (thus on certain compact sets $A\subset \mathbb{R}$ with $m_1(A)=0$.) 
\begin{theorem}[Main Theorem]\label{thm:main}Suppose 
$\Gamma=I\setminus\cup_{j=1}^{\infty}I_j$ is a cut-out set with the following two properties:
\begin{enumerate}[\label=Property I.]
    \item (Measurability) The sequence $(l_j)_{j=1}^{\infty}$ of non-increasing lengths of the complementary intervals of $\Gamma$ has strongly regular decay so that 
$$0<\lim_{j\to\infty}l_jj^{1/d}=L<\infty,$$ for some $d\in(0,1).$ 
    \item (Monotonicity) $\Gamma$ is monotone.
\end{enumerate}
Then 
$$\lim_{\varepsilon\to0^+}N(\Gamma,\varepsilon)\varepsilon^d=L^d\sum_{k=1}^{\infty}\frac{k^d-(k-1)^d}{k}.$$
\end{theorem}
In Section \ref{sec:known} we will explain how Property I is equivalent to Minkowski measurability (Theorem \ref{thm:meas}).
\par The next theorem illustrates that Property I is \textit{necessary} for the existence of positive, finite, first-order packing asymptotics, in at least one case.  
\begin{theorem}\label{thm:sharp1} There exists $s\in(0,1),$ and a cut-out set $S$ satisfying Property II but not Property I, such that 
$$0<\liminf_{\varepsilon\to 0^+}N(S,\varepsilon)\varepsilon^{s}<\limsup_{\varepsilon\to 0^+}N(S,\varepsilon)\varepsilon^{s}<\infty.$$
\end{theorem}
Below, Sharpness Theorem \ref{thm:sharp2} demonstrates that without Property II (Monotonicity) there exists a set satisfying Property I such that $N(\Gamma,\varepsilon)$ converges at the first-order but not to the same limit prescribed by Theorem \ref{thm:main}.
\begin{theorem}\label{thm:sharp2} There exists $t\in(0,1),$ and a cut-out set $T$ satisfying Property I but not Property II, such that
$$0<\lim_{\varepsilon\to 0^+}N(T,\varepsilon)\varepsilon^{t}<L^{t}\sum_{k=1}^{\infty}\frac{k^{t}-(k-1)^{t}}{k}.$$
\end{theorem}
\par Although our results concern only subsets of $\mathbb{R},$ there is a rich history to the study of packing functions and best packing in general. This is because the task of finding exact optimizers for the packing problem, even on geometrically simple sets, can be incredibly vexing (see the \textit{Tammes Problem} \cite{PML1930}, \cite{Musin2015} for the progress on the sphere $\mathbb{S}^2$). Despite these difficulties, there are robust results that describe the first-order asymptotic behavior of $N(A,\varepsilon)$ as $\varepsilon\to 0^+$ (equivalently $\delta(A,N)$ as $N\to\infty$) on a broad range of compact sets $A$. We detail the known results in the next two sections, beginning with the asymptotic growth order of $N(A,\varepsilon)$.
\end{subsection}
\begin{subsection}{Known Asymptotic Results: Growth of $N(A,\varepsilon)$ and Minkowski measurability}
The packing function $N(A,\varepsilon)$ of a compact set $A\subset\mathbb{R}^p$ grows at a rate controlled by its upper and lower  \textit{Minkowski contents}.
\begin{definition}[Minkowski dimension and content]\label{def: mink} Let $0<d\leq p,$ and fix $r>0.$ Then denote the closed \textit{$r$-ball} about $x\in\mathbb{R}^p$ by 
$B_r(x):=\{y\in\mathbb{R}^p\text{ : }||y-x||_p\leq r\}.$ For the remainder of the manuscript, let
\begin{align*}
    \beta_{d}:=
\begin{cases}
    \frac{\pi^{d/2}}{\Gamma(d/2+1)}, & \text{ if } d\in\mathbb{N} \\
    1, & \text{ otherwise }
\end{cases}
\end{align*}
where $\Gamma(z)$ is the Gamma function. Observe that $\beta_p=m_p(B_1(0))$ if $B_1(0)\subset\mathbb{R}^p$ is the $p$-dimensional unit ball. Next, define the closed \textit{$r$-neighborhood} of a $A\subset\mathbb{R}^p$ by
$${B}_r(A):=\{y\in \mathbb{R}^p: \mathrm{dist}(y,A)\leq r\},$$
where $\mathrm{dist}(y,A):=\inf\{||x-y||_p: x\in A\}.$ 
\par 
For each compact $A\subset\mathbb{R}^p,$ we then define the \textit{lower} and \textit{upper $d$-dimensional Minkowski contents of $A$} by 
\begin{align*} \underline{\mathcal{M}}_d(A):=\liminf_{\varepsilon \to 0^+}\frac{m_p\left(B_{\varepsilon}(A)\right)}{\beta_{p-d}\varepsilon^{p-d}},\quad \text{ and } \quad  \overline{\mathcal{M}}_d(A):=\limsup_{\varepsilon \to 0^+}\frac{m_p\left(B_{\varepsilon}(A)\right)}{\beta_{p-d}\varepsilon^{p-d}},
\end{align*}
respectively. If $0<\underline{\mathcal{M}}_d(A)\leq\overline{\mathcal{M}}_d(A)<\infty,$ then we say $A$ has \textit{Minkowski dimension $d$} and write $\dim_M(A)=d.$ If $\dim_M(A)=d,$ and $\underline{\mathcal{M}}_d(A)=\overline{\mathcal{M}}_d(A),$ then we say that $A$ is \textit{Minkowski measurable of dimension $d$}, and define the \textit{$d$-dimensional Minkowski content of $A$} by $\mathcal{M}_d(A):=\underline{\mathcal{M}}_d(A)=\overline{\mathcal{M}}_d(A).$
\end{definition}
\begin{remark} Notice that for any compact $A\subset\mathbb{R}^p,$ $\mathcal{M}_p(A)=m_p(A),$ because $m_p(A(\varepsilon))\to m_p(A)$ trivially as $\varepsilon\to 0^+.$ Also, the upper and lower Minkowski contents of a bounded set and its closure are equal. Hence, much of the theory we describe carries over to bounded Borel sets.
\end{remark}
\begin{definition}[Packing constant] Suppose $A\subset\mathbb{R}^p$ is compact and $d\in(0,p].$ Define the \textit{upper and lower $d$-dimensional packing limits} of $A$ by 
$$\mathcal{\underline{N}}_d(A):=\liminf_{\varepsilon\to 0^+}N(A,\varepsilon)\varepsilon^d, \quad \text{ and }\quad \mathcal{\overline{N}}_d(A):=\limsup_{\varepsilon\to 0^+}N(A,\varepsilon)\varepsilon^d,$$
respectively. If $0<\mathcal{\underline{N}}_d(A)=\mathcal{\overline{N}}_d(A)<\infty,$ then we refer to the common limit $\mathcal{N}_d(A):=\mathcal{\underline{N}}_d(A)=\mathcal{\overline{N}}_d(A)$ as the \textit{packing limit}, or \textit{packing constant}, of $A.$
\end{definition}
Observe that 
$\delta(A,N)N^{1/d}$ converges as $N\to\infty$ if and only if $N(A,\varepsilon)\varepsilon^d$ converges as $\varepsilon\to 0^+.$ Moreover, in the case of convergence:
$$\lim_{\varepsilon\to 0^+}N(A,\varepsilon)\varepsilon^d=\lim_{N\to\infty}\delta(A,N)^d N.$$ We will harmlessly rephrase several results from the literature in terms of $N(A,\varepsilon)$ by using the above relationship.
\begin{proposition}[E.g. \cite{Mattila_2004} Chapter 5, \cite{ARVW2022} Lemma 17, \cite{Borodachov_Hardin_Saff_2007} Proposition 2.5, Inequalities 2.7 and 2.8.]\label{thm:growth} Suppose $A\subset\mathbb{R}^p$ is compact and $d\in(0,p].$ Then
$$0<\mathcal{\underline{N}}_d(A)\leq\mathcal{\overline{N}}_d(A)<\infty,$$
if and only if 
$$0<\mathcal{\underline{M}}_d(A)\leq\mathcal{\overline{M}}_d(A)<\infty.$$
\end{proposition}
The above result follows if there exist positive, finite constants $C_1(p,d),$ and $C_2(p,d),$ depending only on $p,$ and $d,$ such that 
\begin{align}\label{eq:growth}
    C_1(p,d)\frac{m_p(A(\varepsilon))}{\beta_{p-d}\varepsilon^{p-d}}\leq N(A,\varepsilon)\varepsilon^d\leq C_2(p,d)\frac{m_p(A(\varepsilon/2))}{\beta_{p-d}(\varepsilon/2)^{p-d}},
\end{align}
for all $0<\varepsilon\leq \text{diam}(A).$ We provide the proof of the above in Section \ref{sec:lemmas} and
obtain the next result as a corollary of that proof.
\begin{corollary}\label{cor:growth} Suppose $A\subset\mathbb{R}$ is an infinite compact set of Minkowski dimension $d\in(0,1).$ Then
$$\frac{\mathcal{\underline{M}}_d(A)}{4}\leq \mathcal{\underline{N}}_d(A)\leq\mathcal{\overline{N}}_d(A)\leq2^{d-1}\mathcal{\overline{M}}_d(A).$$ 
\end{corollary}
\par Next, we describe how for cut-out sets, Property I of Theorem \ref{thm:main} is equivalent to Minkowski measurability at a given Minkowski dimension $d\in(0,1)$.
\par To prove the one-dimensional Weyl-Barry Conjecture, (see \cite{Lapidus1993}) Lapidus and Pomerance discovered a characterization of Minkowski measurability among boundaries of open subsets of $\mathbb{R}$ that applies to cut-out sets. Falconer later used a dynamical systems argument to concisely prove the result we use below. 
\begin{theorem}[Falconer \cite{Falconer_1995}]\label{thm:meas} Fix $d\in(0,1),$ and suppose $\Gamma=I\setminus\bigcup_{j=1}^{\infty}I_j$ is a cut-out set, with sequence of gap lengths $(l_j)_{j=1}^{\infty}$. Further, let 
$$\underline{L}=\liminf_{j\to\infty}l_jj^{1/d}, \quad \text{ and } \quad \overline{L}=\limsup_{j\to\infty}l_jj^{1/d}.$$ Then
\begin{enumerate}
    \item $\Gamma$ has  Minkowski dimension $d$ if and only if $$0<\underline{L}\leq\overline{L}<\infty.$$
    \item\label{char2} $\Gamma$ is Minkowski measurable of dimension $d$ if and only if $$0<\underline{L}=\overline{L}<\infty.$$ 
\end{enumerate} In Case \ref{char2}, if we denote the common limit by $L=\lim_{j\to\infty}l_jj^{1/d}=\underline{L}=\overline{L},$ then the Minkowski content of $\Gamma$ satisfies  
$$\mathcal{M}_d(\Gamma)=\frac{2^{1-d}}{1-d} L^d.$$
\end{theorem}
\begin{remark}\label{re}
    Notice that the formula for $\mathcal{M}_d(\Gamma)$ depends only on the sequence of gap lengths $(l_j)_{j=1}^{\infty},$ and not the geometric arrangement of the intervals $I_j\subset I$. As presented in the proof of Proposition 2 from \cite{Falconer_1995} (Equation 16); for any $\varepsilon\in(0,l_1/2],$ there exists $n\geq 1$ such that $\varepsilon\in \left[l_{n+1}/2,l_n/2\right],$ and 
$$m_1(\Gamma(\varepsilon))=\sum_{j=n+1}^{\infty}l_j+2(n+1)\varepsilon.$$ Hence, for any $\Gamma'\in\mathcal{R}(\Gamma),$ $m_1(\Gamma'(\varepsilon))=m_1(\Gamma(\varepsilon)).$ Thus, the upper and lower Minkowski contents of cut-out sets are rearrangement invariant.
\end{remark}  
\begin{example}\label{exa}  
Let $\zeta(s)$ be the classical Riemann zeta function (which is meromorphic over $s\in\mathbb{C},$ with a single simple pole at $s=1$) and fix $d\in(0,1)$. Then the set
$$\Gamma_d:=\left\{\sum_{j=1}^{k}\frac{1}{j^{1/d}}:k\geq 1\right\}\cup\ \left\{\zeta\left({1/d}\right)\right\}$$  clearly satisfies Property II (Monotonicity). $\Gamma_d$ also satisfies Property I, with $L=1;$ thus $\Gamma_d$ is Minkowski measurable by Theorem \ref{thm:meas}, and the Main Theorem applies to $\Gamma_d.$
\end{example}
Notice that the example set $S$ considered in Theorem \ref{thm:sharp1} has Minkowski dimension $s\in(0,1),$ but is not Minkowski measurable. The example set $T$ provided in Theorem \ref{thm:sharp2} is Minkowski measurable of dimension $t\in(0,1).$ 
\par 
In the next section, we describe the known results from the literature that provide necessary or sufficient conditions for the existence of $\mathcal{N}_d(A).$ We compare and contrast these with our main results; Theorems \ref{thm:main}, \ref{thm:sharp1}, and \ref{thm:sharp2}, throughout.
\end{subsection}
\begin{subsection}{Known Asymptotic Results: Existence of Packing Constants}\label{sec:known} 
Kolmogorov and Tikhomirov studied the first-order asymptotics of $N(A,\varepsilon)$ to better understand the connections to information theory in \cite{kolmogorov1959}. From this perspective, a set of $N(A,\varepsilon)$ points from $A$ with separation at least $\varepsilon$ provides a system of reliably different signals with which one can fix for transmission or store any binary signal of length less than $\log_2(N(A,\varepsilon));$ the $\varepsilon$-\textit{capacity} of $A$. 
\par In Theorem IX of \cite{kolmogorov1959}, the authors proved that
$$\lim_{\varepsilon\to 0^+}N(A,\varepsilon)\varepsilon^p=C_p m_p(A),$$
for a positive constant $C_p$ depending only on $p,$ provided that $A$ is \textit{Jordan measurable}; that is, so long as $m_p(\partial A)=0,$ where $\partial A$ is the topological boundary of $A.$ Notice, if $m_p(A)>0,$ then $0<\mathcal{N}_p(A)<\infty,$ as $A$ is compact.
\par Many decades passed until Borodachov, Hardin, and Saff proved (see \cite{Borodachov_Hardin_Saff_2007}) that the assumption of Jordan measurability is unnecessary for sets of full dimension. To fully describe their result, we first discuss the notion of a \textit{largest Euclidean sphere packing density}.
\begin{definition}[Largest Euclidean sphere packing density $\Delta_p$] Let $\Lambda_p$ denote the set of collections $\mathcal{P}$ of non-overlapping unit balls in $\mathbb{R}^p$ such that the density 
$$\rho(\mathcal{P}):=\lim_{r\to\infty}(2r)^{-p}m_p\left(\bigcup_{B\in\mathcal{P}} B\cap [-r,r]^p\right)$$
exists. Then the \textit{largest Euclidean sphere packing density} is defined by 
$$\Delta_p:=\sup_{\mathcal{P}\in\Lambda_p}\rho(\mathcal{P}).$$
\end{definition} 
The values of $\Delta_p$ are only known for $p=1,2,3,8,$ and $24$ (\cite{Fejes_1942}, \cite{Hales_2005},\cite{Viazovska_2017},\cite{Cohn_Kumar_Miller_Radchenko_Viazovska_2017}). Attesting to the difficulty of this problem, we note that Kepler conjectured the correct value of $\Delta_3$ in 1611, but this fact was not verified until the $21^{st}$ century by a computer-assisted proof of Hales.  
\par Our new results contrast with the following weak form of Theorem 2.1 from \cite{Borodachov_Hardin_Saff_2007}.
\begin{theorem}[Borodachov, Hardin, and Saff \cite{Borodachov_Hardin_Saff_2007} Theorem 2.1]\label{thm:full} Suppose $A$ is an infinite compact subset of $\mathbb{R}^p.$ Then
$$\lim_{\varepsilon\to 0^+}N(A,\varepsilon)\varepsilon^p=C_p m_p(A)$$  
where $C_p=2^p\Delta_p/\beta_p.$ Moreover, if $m_p(A)>0,$ and $(\omega_N)_{N=1}^{\infty}$ is a sequence of optimizers for $\delta(A,N),$ then the sequence of counting measures
$\nu_N:=\frac{1}{N}\sum_{x\in \omega_N}\delta_x$
converges to the normalized Lebesgue measure $m_p(A\cap \cdot)/m_p(A)$ in the weak$^*$-topology of Borel probability measures on $A.$
\end{theorem}
\begin{remark} For all cut-out sets $\Gamma=I\setminus\cup_{j=1}^{\infty}I_j\subset\mathbb{R},$ $m_1(\Gamma)=m_1(I)-\sum_{j=1}^{\infty}m_1(I_j)=0,$ so $\lim_{\varepsilon\to 0^+}N(\Gamma,\varepsilon)\varepsilon=0,$ and the above distribution result does not apply. Thus, Theorem \ref{thm:full} is mostly unhelpful for describing the asymptotic behavior of $N(\Gamma,\varepsilon)$. 
\end{remark} 
\begin{remark}
The strong form of Theorem 2.1 applies to certain smoother, lower-dimensional sets of measure $0$ that are Minkowski measurable. 
\par Indeed, suppose $d$ is a positive integer at most $p,$ and let $\mathcal{H}_d$ denote the $d$-dimensional Hausdorff measure on $\mathbb{R}^p,$ normalized so that $\mathcal{H}_d([0,1]^d)=1$ for any isometric $d$-dimensional unit cube $[0,1]^d\subset\mathbb{R}^p.$ Then (see Federer \cite{Fed96} Theorem 3.2.39) $\mathcal{H}_d(A)=\mathcal{M}_d(A)$ for all \text{$d$-rectifiable} sets $A\subset \mathbb{R}^p$. $A$ is \textit{$d$-rectifiable} if it is the image of a compact subset of $\mathbb{R}^d$ under a Lipschitz map. The strong form of Theorem 2.1 of \cite{Borodachov_Hardin_Saff_2007} provides that if $\mathcal{H}_d(A)<\infty,$ $\mathcal{H}_d(A)=\mathcal{M}_d(A),$ and $A$ is a countable union of $d$-rectifiable sets together with an $\mathcal{H}_d$ null set, then $$\lim_{\varepsilon\to 0^+}N(A,\varepsilon)\varepsilon^d=C_d\mathcal{H}_d(A).$$ Moreover, when $\mathcal{H}_d(A)>0,$ the optimizing measures $\nu_N$ (as defined in Theorem \ref{thm:full}) converge weakly to $\mathcal{H}_d(A\cap \cdot)/\mathcal{H}_d(A)$ on $A.$ Because the cut-out sets $\Gamma$ that we study are subsets of $\mathbb{R},$ the strong form of Theorem 2.1 does not provide any immediate strengthening over Theorem \ref{thm:full} in this context. However, the Minkowski measurability assumption further motivates our current study.
\end{remark} The proofs of Theorem IX and Theorem \ref{thm:full} fundamentally rely upon analysis of the special case $A=[0,1]^p$ (particularly the  self-similarity of the cube), as well as measure theoretic additivity properties of the set functions  
$$\overline{g}_p(E):=\left(\limsup_{N\to\infty}\delta(E,N)N^{1/p}\right)^p, \quad \text{and}\quad  \underline{g}_p(E):=\left(\liminf_{N\to\infty}\delta(E,N)N^{1/p}\right)^p,\quad E\subset\mathbb{R}^p.$$ The authors of Theorem 15 then obtain their extension of Theorem IX by performing a limiting argument on an analogous result for discrete minimal Riesz $s$-energy (see \cite{HARDIN2005} Theorem 2.1). Indeed, as $s\to\infty,$ the minimal Riesz $s$-energy problem approximates the best packing problem. For a modern exposition of the connections to Riesz $s$-energy, we recommend \cite{BHS2019}.  
\par 
Some results also address the asymptotics of $N(A,\varepsilon)$ on broader classes of sets without establishing the existence of $\mathcal{N}_d(A)$. For example, in Theorem 3 of \cite{Boro2012}, Borodachov proved that for all sets of finite $d$-dimensional packing premeasure $P^d(A)<\infty$, $0<d\leq p,$ the equality  $\overline{g}_d(A)=\nu_{\infty,d}(A)$ holds for a non-trivial outer measure $\nu_{\infty,d}$ obtained via the method I construction (the definitions of $P^d$ and  $\nu_{\infty,d}$ are contained therein). This result \textit{does not}, however, establish the convergence equality $\underline{g}_{d}(A)=\overline{g}_{d}(A)=\nu_{\infty,d}(A)$. 
\par Nonetheless, smoothness, in particular rectifiability, is not always necessary for the convergence of first-order packing asymptotics. To see this, we discuss the known results for packing on \textit{self-similar fractals}.
\begin{definition}[Self-similar fractals]\label{def:self-similar} A mapping $\psi:\mathbb{R}^p\to\mathbb{R}^p$ is a \textit{similitude} if $\psi(x)=rOx+z,$ for a \textit{contraction ratio} $0<r<1,$ an orthogonal matrix $O\in\mathcal{O}(p),$ and a translation $z\in\mathbb{R}^p.$ Each finite set of similitudes $\{\psi_i\}_{i=1}^{M}$ on $\mathbb{R}^p$ determines a unique, non-empty, compact \textit{invariant} set
$$A=\bigcup_{i=1}^{M}\psi_i(A)$$
(see \cite{hutch} Section 3.1, or alternatively \cite{Falconer_1995} Section 9.1).
If $A\subset\mathbb{R}^p$ is invariant under a set of similitudes $\{\psi_i\}_{i=1}^{M},$ and $\psi_i(A)\cap\psi_j(A)=\emptyset$ for each $i\neq j,$ then we say that $A$ is a \textit{self-similar fractal}.\footnote{The requirement that $\psi_i(A)\cap\psi_j(A)=\emptyset$ can often be replaced by a weaker separation condition known as the \textit{open set condition} in many of the results we shall describe.} 
\end{definition}
It is well known that the Hausdorff dimension and Minkowski dimension of a self-similar fractal $A\subset\mathbb{R}^p$ with contractions $\{r_i\}_{i=1}^{M}$ are the same and this common dimension $\dim_M(A)=d\leq p$ satisfies \textit{Moran's equation} depicted below (see \cite{Moran_1946} for Moran's original proof for Hausdorff dimension):
\begin{align}\label{eq:dimensionssf}
    \sum_{i=1}^{M}r_i^d=1.
\end{align} 
Specifically, $0<\mathcal{H}_d(A)<\infty,$ and $0<\mathcal{\underline{M}}_d(A)\leq\mathcal{\overline{M}}_d(A)<\infty$ (see \cite{hutch}). Moran's equation continues to profoundly impact fractal geometry nearly 80 years after its discovery. We recommend \cite{Martinez} for a modern re-exploration of this cornerstone result.  
\par 
The normalized packing $N(A,\varepsilon)\varepsilon^d$ of a self-similar fractal converges as $\varepsilon\to 0^+$ if and only if the set of contractions is multiplicatively independent. We elaborate on this fact below.
\begin{definition}\label{def:indep} A self-similar fractal $A\subset\mathbb{R}$ with contraction ratios $\{r_i\}_{i=1}^{M}$ is \textit{independent} if the set 
$$\left\{\sum_{i=1}^{M}a_i\log(r_i)\text{ : } a_i\in\mathbb{Z}, i=1,\ldots,M \right\}$$
is dense in $\mathbb{R}.$ Otherwise, the above set is an integer lattice: that is, there exists some $h>0$ such that
$$\left\{\sum_{i=1}^{M}a_i\log(r_i)\text{ : } a_i\in\mathbb{Z}, i=1,\ldots,M \right\}=h\mathbb{Z}.$$ In this case we say $A$ is \textit{dependent}. 
\end{definition}   
Lalley used a probabilistic method known as \textit{renewal theory} to show that independent self-similar fractals, which are purely unrectifiable, have convergent first-order packing asymptotics \cite{Lalley}. There, he also established the existence of packing asymptotics along explicit subsequences on dependent self-similar fractals. Later, the first author and Reznikov built upon this to show that dependent self-similar fractals do not have convergent first-order packing asymptotics (\cite{Anderson_Reznikov_2023} Theorem 4.1). Relatedly,
Corollary 1 Equation (18) of \cite{Boro2012} describes how first-order packing asymptotics evolve on compact subsets of an independent self-similar fractal.
\par Gatzouras also used renewal theory to prove that independent self-similar fractals are Minkowski measurable \cite{Gatzouras_1999}, and Kombrink and Winter showed that dependent self-similar fractals $A\subset\mathbb{R}$ are not Minkowski measurable \cite{Kombrink_Winter_2018b}. Thus, Minkowski measurability, independence, and the existence of first-order packing asymptotics are all equivalent on self-similar fractal subsets of $\mathbb{R}.$
\par The renewal theory approach to the packing problem on self-similar fractals is the framework for our proof of Theorem \ref{thm:sharp2}.
\end{subsection}
\begin{subsection}{Further New Results and Examples}\label{sec:main_results} The following asymptotic notations aid with the computations in various results throughout the manuscript.
\begin{definition}[Asymptotic Notations]\label{def:o-not} If $f$ and $g$ are two real-valued functions of a real variable, and $g(x)\geq 0,$ we write
$f(x)\sim g(x),$ as $x\to\infty,$ or $x\to 0^+,$
if $\lim_{x\to\infty}f(x)/g(x)=1,$ or $\lim_{x\to 0^+}f(x)/g(x)=1,$ respectively.
In addition, we write $f(x)\sim 0,$ as $x\to\infty,$ or $x\to 0^+,$ if $\lim_{x\to\infty}f(x)=0,$ or $\lim_{x\to 0^+}f(x)=0,$ respectively.
\par Next, we write
$f(x)=o(g(x)),$ as $x\to\infty$, or $x\to 0^+,$ if 
$f(x)/g(x)\sim 0$ in the corresponding limiting sense.
\par 
Lastly, we shall write $f(x)=O(g(x)),$ as $x\to\infty,$ or $x\to 0^+,$ if there exist constants $0<a\leq b<\infty,$ such that 
$ag(x)\leq f(x)\leq bg(x),$ for all sufficiently large or small $x>0,$ respectively.
Note that these definitions also naturally carry over to infinite sequences $a_n,$ as $n\to\infty$.
\end{definition}
\par Observe that the self-similar sets described in Definition \ref{def:self-similar} are particular examples of cut-out sets $\Gamma=I\setminus \cup_{j=1}^{\infty}I_j.$ If $I$ is the smallest closed interval containing the self-similar fractal $A=\cup_{i=1}^{M}\psi_i(A),$ then  
$$A=\bigcap_{k=1}^{\infty}A_k,$$
where $$A_k=\bigcup_{i_1,\ldots,i_k}\psi_{i_1}\circ\cdots\circ\psi_{i_k}(I),$$
and the union is over all sequences $(i_1,\ldots,i_k)\in\{1,\ldots,M\}^k.$ By reordering the indices of the similarities $\psi_i$ we may further assume that 
$A$ is constructed with $\psi_1(I)\ldots,\psi_M(I)$ subintervals of $I,$ occurring in the same order (from left to right) with $\min\psi_1(I)=\min I,$ and $\max\psi_M(I)=\max I.$ In the sequel, assume without loss of generality that $I=[0,1],$ and refer to the size of the gap between $\psi_{i}(I)$ and $\psi_{i+1}(I)$ by the quantity $b_i$ for each $1\leq i\leq M-1.$
\begin{example}[Montone rearrangement of the 1/3 Cantor set]\label{ternary}
Let $S=[0,1]\setminus\left(\cup_{j=1}^{\infty}I_j\right)$
with intervals $I_j$ satisfying
$m_1(I_j)=l_j=3^{-k},$ for $2^{k-1}\leq j\leq 2^{k}-1,$ $k\geq 1,$ arranged to satisfy Property II. Then notice $A$ is constructed from the same complementary intervals as the 1/3 Cantor set $C$. The 1/3 Cantor set is a dependent self-similar fractal with contractions $r_1=r_2=1/3$ and a singular gap $b_1=1/3,$ so its Minkowski dimension agrees with that of $S$ (Theorem \ref{thm:meas}) and $\dim_M(C)=\dim_M(S)=s$ satisfies
$$\left(\frac{1}{3}\right)^s+\left(\frac{1}{3}\right)^s=2\left(\frac{1}{3}\right)^s=1.$$
Thus $s=\log2/\log3\in(0,1).$\footnote{By $\log x$ we shall always mean the natural logarithm, base $e$} Because $r_1=r_2,$ $C$ is dependent, thus not Minkowski measurable. Hence $S$ is not Minkowski measurable. Using Theorem \ref{thm:meas} it is not difficult to show this directly. We recount the arguments of \cite{Lapidus1993}, Theorem 4.5 below.
\par Since $\sum_{k=1}^{n}2^{k-1}=2^n-1,$ for $n\geq 1,$ we have
$$l_{2^n}=3^{-n-1},\quad \text{ while }\quad l_{2^n-1}=3^{-n}.$$
Therefore
$$\liminf_{j\to\infty}l_j j^{1/s}\leq\lim_{n\to\infty}l_{2^n}(2^n)^{\log_2(3)}=\lim_{n\to\infty}3^{-n-1}3^n=\frac{1}{3},$$
while
$$\limsup_{j\to\infty}l_j j^{1/s}\geq\lim_{n\to\infty}l_{2^n-1}(2^n-1)^{\log_2(3)}=\lim_{n\to\infty}3^{-n}3^n=1>\frac{1}{3}.$$
Thus, by Theorem \ref{thm:meas}, $S$ is \textit{not} Minkowski measurable, and $S$ satisfies Property II but not Property I. 
\end{example}
\begin{example}[The $(1/2,1/3)$ Cantor Set]\label{ex:sharp2}
The example set $T$ of Theorem \ref{thm:sharp2} is the self-similar fractal $T\subset[0,1]$ with similitudes $\psi_1(x)=\frac{1}{2}x,$ and $\psi_2(x)=\frac{1}{3}x+\frac{2}{3},$ and center gap $b_1=1/6.$ Because $\log2/\log3\notin\mathbb{Q},$ $T$ is independent, with convergent packing asymptotics and Minkowski measurable at its Minkowski dimension $t$ solving
$$\left(\frac{1}{2}\right)^t+\left(\frac{1}{3}\right)^t=1.$$
Simultaneously, $T$ is highly non-monotone. Thus, $T$ satisfies Property I but not Property II.
\end{example}
\par Moving on, notice that for each $d\in(0,1)$ the series
$$A_d:=\sum_{k=1}^{\infty}\frac{k^d-(k-1)^d}{k},$$
converges to a finite, positive number. Indeed, $\left({k^d-(k-1)^d}\right)/{k}\sim d k^{d-2}$ as $k\to\infty,$ so the series converges by the Limit Comparison Test. To obtain the limit, one may first apply the Mean Value Theorem to the function $f(x)=x^{d}$ on the interval $[k-1,k],$ extracting a constant $c_k\in[k-1,k]$ such that 
$$\frac{k^d-(k-1)^d}{k}=\frac{f'(c_k)}{k}=\frac{dc_k^{d-1}}{k}.$$
Taking the quotients of the two sides above, then allowing $k$ to tend to infinity establishes the limit.
\par The sets $\Gamma_d$ described in Example \ref{exa} are prototypical of the Main Theorem. Hence, there is the following corollary.
\begin{corollary}\label{cor:first}
For each $d\in(0,1),$ the set $\Gamma_d$ described in Example \ref{exa} satisfies $$\mathcal{N}_d(\Gamma_d)=A_d.$$
\end{corollary}
Due to the formula for the Minkowski content given by Theorem \ref{thm:meas}, we obtain a universality proportionality with $\mathcal{M}_d(\Gamma).$
\begin{corollary}\label{cor:formula} Suppose $\Gamma$ is a cut-out set satisfying Properties I and II, of Minkowski dimension $d\in(0,1),$ and let
$p_d={A_d(1-d)}/{2^{1-d}}.$ Then
$$\mathcal{N}_d(\Gamma)=p_d \mathcal{M}_d(\Gamma).$$
\end{corollary}
\par Sharpness Theorem \ref{thm:sharp2} shows that there cannot be a fractional dimension analog for $C_d$ when $d\in(0,1)$ that is universal with respect to the Minkowski content of the underlying set. Thus, unlike the situation for smoother sets presented in Theorem \ref{thm:full}, when $d\in(0,1)$ there is no way to reverse engineer a definition of $\Delta_d$ from an asymptotic packing identity that is universal with respect to Minkowski content. Nonetheless the next proposition shows that the packing constants $p_d$ of Corollary \ref{cor:formula} approach the expected result $C_1=2^1(\Delta_1/\beta_1)=1$ as $d\to 1^{-}.$
\begin{proposition}\label{thm:limiting} Suppose $p_d=A_d(1-d)/2^{1-d},$ for each $d\in(0,1),$ as provided in Corollary \ref{cor:formula}. Then
$$\lim_{d\to 1^{-}}p_d=\lim_{d\to 1^{-}}\frac{(1-d)}{2^{1-d}}\sum_{k=1}^{\infty}\frac{k^d-(k-1)^d}{k}=1.$$
\end{proposition}
We shall prove the above by using an integral representation of the Digamma function. In fact, the proof amounts to calculating the residue of a particular Dirichlet series at a simple pole.
\par Theorem \ref{thm:meas} shows that Minkowski content is \textit{not} affected by rearrangement of the intervals $I_j$ (this is one of the main ideas in \cite{Lapidus1993}) while Corollary \ref{cor:formula} implies packing asymptotics generally may be. These observations raise natural questions about how the geometric arrangement of the intervals $I_j$ affects the size of these packing constants and their existence. 
\par Indeed, if $\Gamma$ is a cut-out set of Minkowski dimension $d\in(0,1)$, then because the two-sided bounds provided in Corollary \ref{cor:growth} are universal with respect to the upper and lower Minkowski contents, and rearrangement does not affect these quantities (see Remark \ref{re}), the bounds are uniform across all rearrangements of $\Gamma.$ Thus, for such $\Gamma,$ Corollary \ref{cor:growth} gives us the following result.
\begin{proposition} If $\Gamma$ is a cut-out set of Minkowski dimension $d\in(0,1),$ then 
$$0<\frac{\mathcal{\underline{M}}_d(\Gamma)}{4}\leq\inf_{\Gamma'\in\mathcal{R}(\Gamma)}\mathcal{\underline{N}}_d(\Gamma')\leq \sup_{\Gamma'\in\mathcal{R}(\Gamma)}\mathcal{\overline{N}}_d(\Gamma')\leq 2^{d-1}\mathcal{\overline{M}}_d(\Gamma)<\infty.$$    
\end{proposition}
Theorem \ref{thm:main}, Sharpness Theorem \ref{thm:sharp2}, and the above, now lead to the following two corollaries.
\begin{corollary}\label{cor:sharp3} Suppose $\Gamma$ is a cut-out set of Minkowski dimension $d\in(0,1),$ satisfying Property I. Then 
$$p_d\mathcal{M}_d(\Gamma)\leq \sup_{\Gamma'\in\mathcal{R}(\Gamma)}\mathcal{\overline{N}}_d(\Gamma')\leq 2^{d-1}\mathcal{M}_d(\Gamma).$$ 
\end{corollary}
\begin{corollary} There exists $t\in(0,1),$ and a cut-out set $T$ satisfying Property I but not Property II, such that 
    $$0<\inf_{\Gamma\in\mathcal{R}(T)}\mathcal{\underline{N}}_t(\Gamma)<\sup_{\Gamma\in\mathcal{R}(T)}\mathcal{\overline{N}}_t(\Gamma)<\infty.$$
\end{corollary}
The results above lead us to formulate an abstraction of the best packing problem over non-increasing, positive, sequences $(l_j)_{j=1}^{\infty}$ satisfying the dimensionality requirement in Part (1) of Theorem \ref{thm:meas}; i.e. such that
$$0<\underline{L}=\liminf_{j\to\infty}l_jj^{1/d}\leq \limsup_{j\to\infty}{l_jj^{1/d}}=\overline{L}<\infty, \quad \text{ for some }d\in(0,1).$$ Due to integral estimation, the series $\sum_{j=1}^{\infty}l_j$ is absolutely convergent. Thus, we can pose the following question.
\begin{question}[Mass distribution problem]\label{def:mass} Given $\varepsilon>0,$ what is the largest number of \textit{subsets} $\gamma_1,\ldots,\gamma_N$ in a partition of $\mathbb{N}$ such that 
$$\sum_{j\in \gamma_i}l_j\geq \varepsilon, \quad \text{ for all } i=1,\ldots, N?$$
\par 
That this maximum exists for all $\varepsilon\leq \sum_{j=1}^{\infty}l_j$ is immediate as the sum is finite. Because the quantity depends only on the sequence of gap lengths, we shall use $\mathcal{L}=(l_j)_{j=1}^{\infty}$ as a short-hand for this sequence and denote the above maximum by $N(\mathcal{L},\varepsilon).$ Moreover, we shall say the \textit{sequence} $(l_j)_{j=1}^{\infty}$ has Minkowski dimension $d\in(0,1)$ if it satisfies Part (1) of Theorem \ref{thm:meas} (as was done in \cite{Lapidus1993}). Naturally, we will also say that the \textit{sequence} $(l_j)_{j=1}^{\infty}$ is Minkowski measurable of dimension $d\in(0,1)$ if it satisfies Part (2) of Theorem \ref{thm:meas}. In the sequel, we shall refer to the quantity $N(\mathcal{L},\varepsilon)$ as the \textit{mass distribution function} for $\mathcal{L},$ and discuss the task of finding $N(\mathcal{L},\varepsilon)$ as solving the \textit{mass distribution problem}. 
\end{question}
\par The mass distribution problem is a generalized inverse for a corresponding infinite \textit{max-min multiple subset sum problem} (Max-Min MSSP) that asks for the value of 
$$\delta(\mathcal{L},N):=\sup_{\gamma_1,\ldots,\gamma_N\subset \mathbb{N}}\min_{i=1,\ldots,N}\sum_{j\in\gamma_i}l_j,$$
where $N\geq 1,$ and $\gamma_1,\ldots,\gamma_N$ form a partition of $\mathbb{N}$ with trivial capacity constraints $\sum_{j\in\gamma_i}l_j\leq\sum_{j=1}^{\infty}l_j,$ for each $i=1,\ldots,N.$ Caprara, Kellerer, and Pferschy studied the general finite version of the  Max-Min MSSP from the perspective of computer science in \cite{Caprara_Kellerer_Pferschy_2000}, classifying this problem as strongly NP-hard. We establish the following crude first-order asymptotic growth bounds for $N(\mathcal{L},\varepsilon).$ 
\begin{proposition}\label{thm:inv} Suppose $\mathcal{L}=(l_j)_{j=1}^{\infty}$ has Minkowski dimension $d\in(0,1)$, with $0<\underline{L}=\liminf_{j\to\infty }l_jj^{1/d}\leq \limsup_{j\to\infty}{l_jj^{1/d}}=\overline{L}<\infty.$ Then
\begin{equation}\label{eq:p1}
  \underline{L}^d A_d\leq \liminf_{\varepsilon\to 0^+}N(\mathcal{L},\varepsilon)\varepsilon^d\leq \limsup_{\varepsilon\to 0^+}N(\mathcal{L},\varepsilon)\varepsilon^d\leq \overline{L}^d+\frac{d}{1-d}\overline{L}\underline{L}^{d-1}.  
\end{equation}
In particular, if $\mathcal{L}$ is Minkowski measurable, then
\begin{equation}\label{eq:p2}
    A_dL^d\leq \liminf_{\varepsilon\to 0^+}N(\mathcal{L},\varepsilon)\varepsilon^d\leq \limsup_{\varepsilon\to 0^+}N(\mathcal{L},\varepsilon)\varepsilon^d\leq \frac{L^d}{1-d}
\end{equation}
where $L=\underline{L}=\overline{L}.$
\end{proposition}
\begin{remark} Note that there are some remarkable similarities between the mass distribution function $N(\mathcal{L},\varepsilon)$ and the function $\delta(1/\varepsilon)=\sum_{j=1}^{\infty}(l_j/\varepsilon-\lfloor l_j/\varepsilon \rfloor),$ (first studied in \cite{Lapidus1993}) in the limit as $\varepsilon\to 0^+.$ Indeed, Theorem 3.2 of \cite{Lapidus1993} shows that these functions have the same growth order as $\varepsilon\to 0^+$, and the asymptotic upper bound for $\limsup_{\varepsilon\to 0^+}\delta(1/\varepsilon)\varepsilon^d$ agrees with the asymptotic upper bound we provide on the mass distribution function. Lapidus and Pomerance in fact show (Theorem 4.2 of \cite{Lapidus1993}) that $\delta(1/\varepsilon)\varepsilon^d\sim -\zeta(d)L^d$ as $\varepsilon\to 0^+$ if $\mathcal{L}$ is Minkowski measurable with $l_jj^{1/d}\sim L,$ for some $0<L<\infty$. Because the Minkowski dimensions $d\in(0,1)$ belong to the \textit{critical strip}, $0<\text{Re}(s)<1,$ of $\zeta(s)$, the study of the geometry of cut-out sets has some intriguing connections to analytic number theory and the Riemann Hypothesis \cite{Lapidus_Maier_1995b}. In detail, Corollary 2.7 of \cite{Lapidus_Maier_1995b} shows how an inverse spectral theorem for cut-out sets of Minkowski dimension $d\in(0,1)$ is false only when $d=1/2$ \textit{if and only if} the
Riemann Hypothesis is true. 
\par Falconer also used the asymptotic result for $\delta(1/\varepsilon)$ to study the following related packing problem (see Section 4 of \cite{Falconer_1995}): Given an open $\Omega\subset\mathbb{R},$ expressed as a countable disjoint union of open intervals $\Omega=\cup_{j=1}^{\infty}I_j,$ where $|I_j|=l_j\sim Lj^{-1/d},$ and $0<L<\infty,$ let $P(\mathcal{L},\varepsilon)$ be the maximum number of disjoint open intervals of radius $\varepsilon$ that one may place in $\Omega.$ Then 
$$\varepsilon^d\left(P(\mathcal{L},\varepsilon)-\frac{1}{\varepsilon}\sum_{j=1}^{\infty}l_j\right)\sim\zeta(d)L^d\quad\text{ as }\quad\varepsilon\to 0^+.$$
\end{remark}
\end{subsection}
\begin{subsection}{Outline}
\par In Section \ref{sec:proofs}, we shall proceed with the proofs. We begin Section \ref{sec:lemmas} by establishing some propositions and lemmas that we shall need to use repeatedly throughout the manuscript.
In Section \ref{sec:main} we provide some linear programming background before proving the asymptotic upper and lower bounds that suffice to complete the Main Theorem. After this, we prove the limiting behavior of the constants $p_d$ in Proposition \ref{thm:limiting} and the growth bounds of Proposition \ref{thm:inv}. Sections \ref{sec:sharp1} and \ref{sec:sharp2} contain the proofs of Theorems \ref{thm:sharp1} and \ref{thm:sharp2}, respectively. In particular, we spend some time in Section \ref{sec:sharp2} describing the results of renewal theory that are relevant to our proof of Theorem \ref{thm:sharp2}.
\par We complete the manuscript by stating two conjectures concerning the relationship between packing asymptotics and Minkowski measurability in Section \ref{sec:CQfD}.
\end{subsection}
\section{Proofs}\label{sec:proofs}
\begin{subsection}{Lemmas}\label{sec:lemmas} We begin the proofs by verifying Proposition \ref{thm:growth}. 
\begin{proof}[Proof of Proposition \ref{thm:growth}] We proceed to establish the sufficient set of inequalities in Line \ref{eq:growth}. To this end, let $\omega_{N(A,\varepsilon)}\subset A$ satisfy $\delta(\omega_{N(A,\varepsilon)})\geq \varepsilon.$ Due to the maximality of $N(A,\varepsilon),$ for any $y\in A\setminus\omega_{N(A,\varepsilon)},$ we have $\text{dist}(y,\omega_{N(A,\varepsilon)})<\varepsilon.$ Thus 
$A\subset \bigcup_{x\in \omega_{N(A,\varepsilon)}}B_{\varepsilon}(x),$ and by the triangle inequality,
 $$A(\varepsilon)\subset \bigcup_{x\in \omega_{N(A,\varepsilon)}}B_{2\varepsilon}(x).$$ 
Therefore, by the monotonicity of measure, $m_p(A(\varepsilon))\leq (2\varepsilon)^p\beta_p N(A,\varepsilon).$ Thus,
$$\frac{\beta_{p-d}}{2^p\beta_p}\cdot \frac{m_p(A(\varepsilon))}{\beta_{p-d}\varepsilon^{p-d}}\leq N(A,\varepsilon)\varepsilon^d,$$
and the lower bound holds for $C_1(p,d):=\beta_{p-d}/(2^p\beta_p).$ 
\par Next, we establish the upper bound. If $\omega_{N(A,\varepsilon)}$ is a packing of radius at least $\varepsilon$, we know that the $\varepsilon/2$-balls around each point of the configuration have disjoint interiors, so $m_p\left(B_{\varepsilon/2}(x)\cap B_{\varepsilon/2}(y)\right)=0,$ for each pair of distinct points $x,y\in \omega_{N(A,\varepsilon)}.$ Since
$$\bigcup_{x\in\omega_{N(A,\varepsilon)}}B_{\varepsilon/2}(x)\subset A(\varepsilon/2),$$  it follows from the disjointness that $m_p(A(\varepsilon/2))\geq ({\varepsilon}/{2})^p \beta_p N(A,\varepsilon)$. Therefore,
$$N(A,\varepsilon)\varepsilon^d\leq \frac{2^d\beta_{p-d}}{\beta_p}\cdot \frac{m_p(A(\varepsilon/2))}{\beta_{p-d}(\varepsilon/2)^{p-d}},$$
and the upper bound holds with $C_2(p,d):=2^d\beta_{p-d}/\beta_p$
\end{proof}
We will frequently use the following decomposition, the proof of which is immediate from the definitions.
\begin{proposition}\label{thm:simpchar} Suppose $\Gamma=I\setminus\cup_{j=1}^{\infty}I_j$ is a cut-out set satisfying Property II with the sequence of gap lengths $(l_j)_{j=1}^{\infty}.$ Then $\Gamma$ may be written in the form
    $$\Gamma=\left\{ y_k:k\geq 1\right\}\cup\ \left\{y_1+\sum_{j=1}^{\infty}l_j  \right\},$$ for a sequence $y_k$ defined by
    $y_1=
    \min I=\inf I_1$, and $y_{k+1}-y_k=l_k,$ for $k\geq 1.$ In particular, $\Gamma$ is countably infinite.
\end{proposition}
We also need the following version of Abel's partial summation formula, which we state and prove for completion below.
\begin{lemma}[Partial summation]\label{combin} Suppose $(a_n)_{n=1}^{\infty}, (b_n)_{n=0}^{\infty}$ are two infinite sequences of real numbers such that either $a_1=0$ or $b_0=0$. For each $N\geq 2,$ define two finite sequences 
$$s_n=a_n\left(b_n-b_{n-1}\right),\quad \text{ for } 1\leq n\leq N,\quad \text{ and }$$ 
$$t_n=b_n\left({a_n}-{a_{n+1}}\right), \quad \text{ for } 1\leq n\leq N-1, \quad t_N={a_N}{b_N}.$$
Then for all $N\geq2,$
\begin{align*}
  \sum_{n=1}^{N}s_n=\sum_{n=1}^{N}t_n  .
\end{align*}
\end{lemma}
\begin{proof}
Let $c_n=a_n b_{n-1}$ for each $1\leq n \leq N,$ and observe that 
\begin{align*}s_n&=a_n(b_n-b_{n-1})=b_n(a_n-a_{n+1})+(a_{n+1}b_n-a_nb_{n-1})\\
&=b_n(a_n-a_{n+1})+(c_{n+1}-c_n)=t_n+(c_{n+1}-c_n),
\end{align*}
for all $1\leq n\leq N-1.$ Further, notice that $c_1=a_1b_0=0,$ as $a_1=0$ or $b_0=0.$ Thus, by the above
\begin{align*}
\sum_{n=1}^{N}s_n&=a_N(b_N-b_{N-1})+\sum_{n=1}^{N-1}s_n\\
&=a_Nb_N-c_N+\sum_{n=1}^{N-1}t_n+\sum_{n=1}^{N-1}(c_{n+1}-c_n)\\
&=a_Nb_N-c_N+c_N-c_1+\sum_{n=1}^{N-1}t_n=\sum_{n=1}^{N}t_n.
\end{align*}
\end{proof} 
Next, we use Theorem \ref{thm:meas} to study the distribution of distances in $d$-dimensional cut-out sets. Suppose $(l_j)_{j=1}^{\infty}$ is a positive sequence such that $\sum_{j=1}^{\infty}l_j<\infty$. Then for each $n\geq 1,$ and $k\geq 1,$ set
$$p(n,k):=\sum_{j=n}^{n+k-1}l_{j}.$$ 
For all $\varepsilon\in(0,l_1]$ define
$$F(k,\varepsilon):=\max\{n\geq 1: p(n,k)\geq \varepsilon\}.$$
\begin{figure}[h]\label{fig:jumps}
\includegraphics[width=16cm]{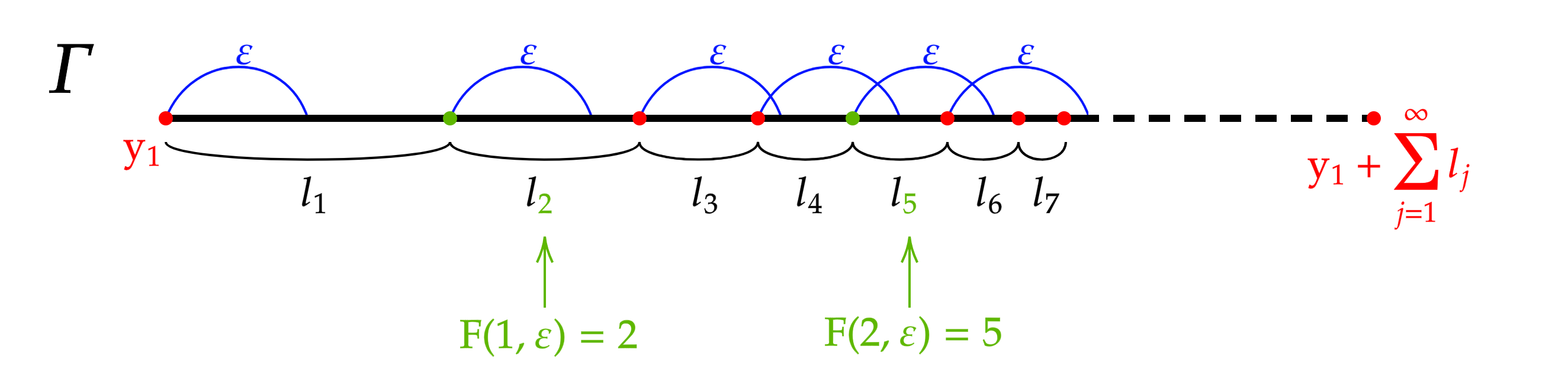}
\caption{If $\Gamma=I\setminus \cup_{j=1}^{\infty}I_j$ is a monotone cut-out set, with gap sequence $(l_j)_{j=1}^{\infty},$ then $F(k,\varepsilon)$ provides control over the set of distances in $\Gamma.$}
\centering
\end{figure}
\begin{lemma}\label{thm:Lemma} Suppose $\mathcal{L}=(l_j)_{j=1}^{\infty}$ has Minkowski dimension $d\in(0,1)$, with $0<\underline{L}=\liminf_{j\to\infty }l_jj^{1/d}\leq \limsup_{j\to\infty}{l_jj^{1/d}}=\overline{L}<\infty.$ 
Assume $u(\varepsilon)$ is a positive, integer-valued function such that
\begin{enumerate}[(i)]
    \item $\lim_{\varepsilon\to 0^+}u(\varepsilon)=\infty,$ 
    \item $u(\varepsilon)=o(\varepsilon^{{d}/{(d-1)}})\text{ as } \varepsilon\to 0^+.$
\end{enumerate}
Then
\begin{equation}\label{eq:part1}
\underline{L}^d\leq \liminf_{\varepsilon\to 0^+}\min_{1\leq k \leq u(\varepsilon)} F(k,\varepsilon)\left(\frac{\varepsilon}{k}\right)^d\leq \limsup_{\varepsilon\to 0^+}\sup_{k\geq 1}F(k,\varepsilon)\left(\frac{\varepsilon}{k}\right)^d\leq \overline{L}^d.  
\end{equation}
In particular, if $\mathcal{L}$ is Minkowski measurable, then
\begin{equation}\label{eq:part2}
  \lim_{\varepsilon\to 0^+}\min_{1\leq k \leq u(\varepsilon)} F(k,\varepsilon)\left(\frac{\varepsilon}{k}\right)^d= \lim_{\varepsilon\to 0^+}\sup_{k\geq 1}F(k,\varepsilon)\left(\frac{\varepsilon}{k}\right)^d=L^d, 
\end{equation}
where $L=\underline{L}=\overline{L}.$
\end{lemma}
\begin{proof}
Clearly, Line \ref{eq:part2} follows from Line \ref{eq:part1}, so we proceed to prove the inequalities listed in Line \ref{eq:part1}. To this end, first set $L_j=l_jj^{{1/d}}.$ Then let $\delta\in(0,\underline{L}/2)$ and observe that there exists a natural $M=M(\delta)$ such that  
$\underline{L}-\delta\leq L_n\leq \overline{L}+\delta, \text{ for all } n\geq M.$
Hence, for all $n\geq M,$ and fixed $k\geq 1,$ we may trivially estimate $p(n,k)$ from above and below to obtain the following two inequalities:
\begin{equation}\label{eq:triv1}
    (\underline{L}-\delta)\frac{k}{(n+k)^{1/d}}\leq \sum_{j=n}^{n+k-1}\frac{(\underline{L}-\delta)}{j^{{1/d}}} \leq p(n,k)\leq \sum_{j=n}^{n+k-1}\frac{(\overline{L}+\delta)}{j^{{1/d}}}\leq (\overline{L}+\delta)\frac{k}{n^{1/d}}, 
\end{equation}
for all $n\geq M.$ Now, for any $\varepsilon>0,$ let
$$q(k,\varepsilon,\delta)=\left\lceil(\overline{L}+\delta)^d\left(k/\varepsilon\right)^d\right\rceil+1, \quad \text{and} \quad  
r(k,\varepsilon,\delta)=\left\lfloor(\underline{L}-\delta)^d\left(k/{\varepsilon}\right)^d-k\right\rfloor.$$
Notice that $q(k,\varepsilon,\delta)$ is non-decreasing in $k$. Thus, if
$0<\varepsilon\leq \varepsilon_1(\delta):=({\overline{L}+\delta})/{(M-1)^{1/d}},$
then
$q(k,\varepsilon,\delta)\geq q(1,\varepsilon,\delta)\geq \left(\left({\overline{L}+\delta}\right)/{\varepsilon}\right)^d+1\geq M.$ Hence, for fixed $\delta>0$ we may take $n=q(k,\varepsilon,\delta)$ in the upper estimate of Line \ref{eq:triv1} to verify that
$p(q(k,\varepsilon,\delta),k)<\varepsilon$ for all $\varepsilon\leq\varepsilon_1(\delta),$ and all $k\geq 1.$ Therefore,
$F(k,\varepsilon)\leq q(k,\varepsilon,\delta)-1\leq (\overline{L}+\delta)^d\left({k}/{\varepsilon}\right)^d+1,$ for all $\varepsilon\leq \varepsilon_1(\delta),$
which implies
$$\sup_{k\geq 1}F(k,\varepsilon)\left(\frac{\varepsilon}{k}\right)^d\leq (\overline{L}+\delta)^d+\varepsilon^d$$
for all $\varepsilon\leq \varepsilon_1(\delta).$ Sending $\varepsilon\to 0^+,$ we obtain
$$\limsup_{\varepsilon\to 0^+}\sup_{k\geq 1}F(k,\varepsilon)\left(\frac{\varepsilon}{k}\right)^d\leq (\overline{L}+\delta)^d.$$ Thus, the upper bound of Line \ref{eq:triv1} follows upon sending $\delta\to 0^+$.
\par We seek to deduce a similar behavior for $r(k,\varepsilon,\delta)$ for all not too large $k$, but this requires more care. For $c>0,$ and $x\geq 1,$ we thus consider an auxiliary function for $r(k,\varepsilon,\delta)$ defined by
$r_c(x):=cx^d-x.$ 
Because $r_c''(x)=cd(d-1)x^{d-2},$ $c>0,$ $x\geq 1,$ and $0<d<1,$ it follows that $r_c(x)$ is strictly concave, and the $x$-coordinate of its maximum $x^{*}:=x^{*}(c)=(cd)^{1/(1-d)}$ grows without bound as $c\to\infty.$ Next, notice that 
$\lfloor r_{c}(k)\rfloor=r(k,\varepsilon,\delta),$ if $c=c(\varepsilon,\delta):=\left(\left(\underline{L}-\delta\right)/\varepsilon\right)^d,$ and let
$k^*(c):=\lfloor x^*(c) \rfloor.$ Observe that $x^*(c)$ is increasing in $c,$ and when $\varepsilon$ is held fixed, $c(\varepsilon,\delta)$ is strictly decreasing for $\delta\in(0,\underline{L}/2).$ Thus, 
$$u_0(\varepsilon):=k^*(c(\varepsilon,\underline{L}/2))\leq k^*(c(\varepsilon,\delta))$$
for all $\delta\in(0,\underline{L}/2).$
Now, notice that for all $k$ in the range
$1\leq k \leq u_0(\varepsilon),$ 
$r(k,\varepsilon,\delta)\geq r(1,\varepsilon,\delta)\geq\left(\left({\underline{L}-\delta}\right)/{\varepsilon}\right)^d-2.$ Hence, if $0<\varepsilon\leq \varepsilon_2(\delta):=\left({\underline{L}-\delta}\right)/{(M+2)^{1/d}},$ then for all $1\leq k\leq u_0(\varepsilon),$ we have $r(k,\varepsilon,\delta)\geq r(1,\varepsilon,\delta)\geq M.$ Because $u_0(\varepsilon)$ is independent of $\delta,$ $u_0(\varepsilon)=O(\varepsilon^{d/(d-1)}),$ and $u(\varepsilon)$ satisfies assumption (ii), there exists $\varepsilon_3>0,$ such that 
$u(\varepsilon)\leq u_0(\varepsilon),$
uniformly for all $\delta<\underline{L}/2,$ and all $\varepsilon\leq \varepsilon_3.$  
\par Thus, we may apply the estimates from Line \ref{eq:triv1} to obtain that
$p(r(k,\varepsilon,\delta),k)\geq \varepsilon,$ for all $\varepsilon\leq \min\left(\varepsilon_2(\delta),\varepsilon_3\right),$ and all $1\leq k\leq u(\varepsilon).$
By the definitions of $F(k,\varepsilon)$ and $r(k,\varepsilon,\delta)$ we obtain
$$F(k,\varepsilon)\geq r(k,\varepsilon,\delta)\geq (\underline{L}-\delta)^d\left(\frac{k}{\varepsilon}\right)^d-k-1$$ for all $\varepsilon\leq \min\left(\varepsilon_2(\delta),\varepsilon_3\right),$ and all $1\leq k\leq u(\varepsilon)$. Rearranging the above,
we have 
\begin{align*}
\min_{1\leq k\leq u(\varepsilon)}F(k,\varepsilon)\left(\frac{\varepsilon}{k}\right)^d&\geq \min_{1\leq k\leq u(\varepsilon)}\left((\underline{L}-\delta)^d-\frac{(k+1)\varepsilon^d}{k^d}\right)\\
&=(\underline{L}-\delta)^d-\varepsilon^d\max_{1\leq k\leq u(\varepsilon)}\frac{k+1}{k^d}\\
&=(\underline{L}-\delta)^d-\max\left(2\varepsilon^d,\frac{(u(\varepsilon)+1)\varepsilon^d}{u(\varepsilon)^d}\right) 
\end{align*}
for all $\varepsilon\leq \min(\varepsilon_2(\delta),\varepsilon_3),$
because $b(x)=(x+1)/{x^d}$ has a unique global minimum at $x_d=d/(1-d),$ for all $d\in(0,1),$ and is increasing for all $x>x_d.$ Thus, assumption (ii) on $u(\varepsilon)$ implies that for each fixed $\delta\in(0,\underline{L}/2),$
$$\min_{1\leq k\leq u(\varepsilon)}F(k,\varepsilon)\left(\frac{\varepsilon}{k}\right)^d\geq (\underline{L}-\delta)^d+o(1)\quad\text{ as }\quad\varepsilon\to 0^+,$$
and the result follows by sending $\varepsilon\to 0^+,$ then $\delta\to0^+.$
\end{proof}
\begin{remark} The growth restriction (ii) on $u(\varepsilon)$ is nearly the best possible based on the lower estimate in Line \ref{eq:triv1}. Indeed, because $r_c(x)$ is decreasing on $[x^*(c),\infty),$ $r_c(x)$ is invertible in this range. Thus, if $x_0$ is the unique solution of $r_c(x_0)=c{x_0}^d-x_0=c-1$ with $x_0>1,$ then $x_0$ is a function of $c$ (we shall write $x_0=x_0(c)$ where prudent) and $x^*(c)\leq x_0(c) \to\infty,$ as $c\to\infty$. Baring this fact in mind, notice that $r_c(x_0)=c{x_0}^d-x_0=c-1$ implies 
$$c=\frac{x_0-1}{{x_0}^d-1}\leq x_0^{1-d}(1+o(1)) \quad\text{ as }\quad c\to \infty.$$ Therefore, $x_0\geq \left(c/(1+o(1))\right)^{1/(1-d)}$ as $c\to\infty.$ In fact, we can use a similar idea to the above to establish that 
$$\lim_{c\to\infty}\frac{x_0(c)}{c^{1/(1-d)}}=1.$$ 
Thus, the largest range of $k$-values where $r(k,\varepsilon,\delta)\geq r(1,\varepsilon,\delta)$ is given by $1\leq k \leq x_0(c(\varepsilon,\delta)),$ where $x_0(c(\varepsilon,\delta))\sim\left(\underline{L}-\delta\right)^{d/(1-d)}\varepsilon^{d/(d-1)}=O(\varepsilon^{d/(d-1)})$ as $\varepsilon\to 0^+.$   
\end{remark}
\end{subsection}
\begin{subsection}{Proof of The Main Theorem}\label{sec:main}
In this section, we provide asymptotic upper and lower bounds on the packing function that suffice to prove Theorem \ref{thm:main}. To prove the upper bound, we will use a duality theorem from linear programming that can be found for example in \cite{yudin}, Corollary 3.3, on page 94. 
\par Linear programming has long been a useful tool in geometric optimization theory. A well-known example of this is the work of Cohn and Elkies in the area of sphere packing. In \cite{Cohn_Elkies_2003} the authors pioneered a linear programming approach to the densest Euclidean sphere packing problem that gives an upper bound on $\Delta_d$ for each integer $d\geq1$. This paved the way for Viazovska's Fields Medal earning resolution of the exact value of $\Delta_8$ \cite{Viazovska_2017}, and later $\Delta_{24}$ with collaborators  \cite{Cohn_Kumar_Miller_Radchenko_Viazovska_2017}.
\par After presenting the proof of the upper bound, we obtain the lower bound constructively, using a greedy algorithm. The two-sided bounds describe a method to produce asymptotically tight bounds between a particular integer linear program and its real input counterpart. 
In the following, we state the linear programming theory that we shall use in both matrix and explicit form for the convenience of the reader.
\begin{definition}(Linear programming)\label{thm:dual} Consider a linear programming  problem as follows:
\par Given a vector of weights $\mathbf{c}\in\mathbb{R}^n,$ and an $m\times n$ matrix $\mathbf{A}=[a_{ij}]\in \mathbb{R}^m\times \mathbb{R}^n,$ maximize the linear form 
\begin{align}\label{eqn: max1}
  \mathbf{c}^T\mathbf{x} =\sum_{j=1}^{n}c_jx_j
\end{align}
among all $\mathbf{x}\in\mathbb{R}^{n},$
subject to the conditions 
\begin{align}\label{eqn: max2}
    \sum_{j=1}^{n}a_{ij}x_j\leq b_i,\text{ for }i=1,\ldots,m,
\end{align}
\begin{align}\label{eqn: max3}
    x_j\geq0, \text{ for }j=1,\ldots,n.
\end{align}
Equivalently, we have the matrix constraints
$$\mathbf{A}\mathbf{x}\leq \mathbf{b}, \quad \text{ and } \quad \mathbf{x}\geq \mathbf{0},$$
where $\textbf{b}=(b_1,b_2,\ldots,b_m)^T\in\mathbb{R}^m.$
A linear programming problem following the above form is termed \textit{primal}. The corresponding \textit{dual linear programming  problem} is the task listed below:
\par 
Minimize the linear form
\begin{align}\label{eqn: min1}
    \mathbf{b}^T\mathbf{y}=\sum_{i=1}^{m}b_iy_i
\end{align}
among all $\textbf{y}\in\mathbb{R}^{m},$ subject to the conditions
\begin{align}\label{eqn: min2}
    \sum_{i=1}^{m}a_{ij}y_i\geq c_j,\text{ for }j=1,\ldots,n,
\end{align}
\begin{align}\label{eqn: min3}
    y_i\geq 0,\text{ for }i=1,\ldots,m.
\end{align}
Equivalently, there are the matrix constraints 
\begin{equation}\label{eq:min4}
\mathbf{A}^T\mathbf{y}\geq \mathbf{c}, \quad \text{ and } \quad \mathbf{y}\geq \mathbf{0}.
\end{equation}
We refer to a vector $\mathbf{x^{*}}=(x_1^{*},\ldots,x_n^{*})^T$ satisfying (\ref{eqn: max2})-(\ref{eqn: max3}) as a \textit{feasible solution} to the corresponding linear programming problem. 
\end{definition}
\begin{theorem}[Strong Duality]\label{thm:strongdual}
For any dual pair of linear programming  problems given by (\ref{eqn: max1})-(\ref{eqn: max3}) and (\ref{eqn: min1})-(\ref{eqn: min3}), two feasible programs ${\mathbf{x}^{*}}=(x_1^{*},\ldots,x_n^{*})^T,$ and $\mathbf{y}^{*}=(y_1^{*},\ldots,y_m^{*})^T,$ are respectively optimal for the primal problem and its dual if and only if 
$$\mathbf{c}^T\mathbf{x}^{*}=\sum_{j=1}^{n}c_jx_j^{*}=\sum_{i=1}^{m}b_iy_i^{*}=\mathbf{b}^T\mathbf{y}^{*}.$$
\end{theorem}
\begin{theorem}[Upper Bound]\label{thm:upper} Fix $d\in(0,1),$ and $0<L<\infty.$ Assume $\Gamma=I\setminus\cup_{j=1}^{\infty}I_j$ is a cut-out set of Minkowski dimension $d$ satisfying Properties I and II, with $l_jj^{1/d}\sim L$ as $j\to\infty.$ Then 
$$\limsup_{\varepsilon\to0^{+}}N(\Gamma,\varepsilon)\varepsilon^d\leq {L}^d A_d.$$
\end{theorem}
\begin{proof}[Proof of Theorem \ref{thm:upper}]
Let $\omega_{P(\varepsilon)}$ be a configuration of $P(\varepsilon)$ points within $\Gamma,$ with separation $\delta\left(\omega_{P(\varepsilon)}\right)\geq \varepsilon.$ Using Proposition \ref{thm:simpchar}, we may order the elements of $\omega_{P(\varepsilon)}$ according to 
$$y_1=0\leq x_1<x_2<\cdots<x_{P(\varepsilon)}\leq m_1(I),$$ 
where we have chosen $y_1=0,$ without loss of generality.
Further, for each $1\leq n \leq P(\varepsilon)-1,$ Proposition \ref{thm:simpchar} provides a subsequence $j_n,$ for which $x_n=y_{j_n}.$ If $k_n=j_{n+1}-j_n,$ we may then represent the sequence of (adjacent) distances in $\omega_{P(\varepsilon)}\setminus\{x_{P(\varepsilon)}\}$ by 
$$x_{n+1}-x_n=y_{j_{n+1}}-y_{j_n}=\sum_{j=j_{n}+1}^{j_{n+1}}l_j=\sum_{j=j_n+1}^{j_n+k_n}l_j.$$ Our first task is to reduce the problem of upper bounding $N(\Gamma,\varepsilon)$ to the problem of counting the maximal number of occurrences of sums of $k$ terms in the set of adjacent distances in $\Gamma$. 
For each $k\geq 1,$ we thus define the frequency
$$f(k,\omega_{P(\varepsilon)})=\#\left\{1\leq n\leq P(\varepsilon)-1: k_n=k \right\}.$$ 
Let $$K(\varepsilon)=\min\left\{n\geq 1: \sum_{j=n}^{\infty}l_j\leq \varepsilon\right\},$$
and observe that
\begin{align}\label{eq:F(K)}
F(k,\varepsilon)\leq K(\varepsilon),\text{ 
uniformly, for all $k\geq 1.$}
\end{align}
 Indeed, if not, then we can find $k'\geq 1$ such that $F(k',\varepsilon)\geq K(\varepsilon)+1,$ and
$$\varepsilon\leq \sum_{j=F(k',\varepsilon)}^{F(k',\varepsilon)+k'-1}l_j\leq \sum_{j=K(\varepsilon)+1}^{\infty}l_j=-l_{K(\varepsilon)}+\sum_{j=K(\varepsilon)}^{\infty}l_j<\varepsilon,$$ a clear contradiction. But it is also clear that 
$$\#\left(\omega_{P(\varepsilon)}\cap \left(\{y_n\in\Gamma\text{ : } n\geq K(\varepsilon)\}\cup \{y_{1}+\sum_{j=1}^{\infty}l_j\}\right)\right)\leq 2,$$
by the definition of $K(\varepsilon)$ and the $\varepsilon$-separation of $\omega_{P(\varepsilon)}.$ Thus, 
$$P(\varepsilon)-2\leq \sum_{k=1}^{\infty}f(k,\omega_{P(\varepsilon)}).$$
Since $\Gamma$ is monotone and the quantity $jf(j,\omega_{P(\varepsilon)})$ represents the total number of integers in the range $1,2,\ldots,F(j,\varepsilon)+j-1$ that are exhausted by $j$-term progressions, 
\begin{equation}\label{eq:jumps1}\sum_{j=1}^{k}jf(j,\omega_{P(\varepsilon)})\leq F(k,\varepsilon)+k-1
\end{equation}
for all $k\geq 1.$ Next we claim that $
\sum_{k=K(\varepsilon)}^{\infty}f(k,\omega_{P(\varepsilon)})\leq 1.$ We accomplish this by showing that if there exists an integer $k'\geq K(\varepsilon)$ such that $f(k',\omega_{P(\varepsilon)})\geq 1,$ then $f(k'',\omega_{P(\varepsilon)})=0$ for any $k''\geq k'$. This ensures that the total sum over all $k\geq K(\varepsilon)$ is always at most $1.$ Assuming $f(k',\omega_{P(\varepsilon)})\geq 1,$ we obtain
\begin{align*}
  k'+k''f(k'',\omega_{P(\varepsilon)})&\leq k'f(k',\omega_{P(\varepsilon)})+k''f(k'',\omega_{P(\varepsilon)})\\
  &\leq F(k'',\varepsilon)+k''-1\\
  &\leq K(\varepsilon)+k''-1\\
  &\leq k'+k''-1,
\end{align*}
by Lines \ref{eq:F(K)} and \ref{eq:jumps1}. Thus $f(k'',\omega_{P(\varepsilon)})\leq(k''-1)/k''<1,$ which implies $f(k'',\omega_{P(\varepsilon)})=0,$ as $f(k'',\omega_{P(\varepsilon)})\in\mathbb{Z}.$ From these deductions, observe that 
$$P(\varepsilon)-3\leq \sum_{k=1}^{K(\varepsilon)}f(k,\omega_{P(\varepsilon)}).$$
Thus, we have reached the desired reduction. 
\par 
Next, let $\delta>0.$ Then by Lemma \ref{thm:Lemma}, we may find $\varepsilon_0=\varepsilon_0(\delta)>0$ such that for all $0<\varepsilon\leq \varepsilon_0,$ and for all $k\geq 1,$
\begin{equation}\label{eq:jumps2}
    F(k,\varepsilon)\leq (L+\delta)^d\left(\frac{k}{\varepsilon}\right)^d.
\end{equation}
Therefore, for each $1\leq k\leq K(\varepsilon),$ we may put together Lines \ref{eq:jumps1} and \ref{eq:jumps2} to obtain
\begin{equation*}\sum_{j=1}^{k}jf(j,\omega_{P(\varepsilon)})\leq F(k,\varepsilon)+k-1\leq (L+\delta)^d\left(\frac{k}{\varepsilon}\right)^d+k-1.
\end{equation*}
Next, let $f_k=f(k,\omega_{P(\varepsilon)}),$ for each $1\leq k\leq K(\varepsilon).$ We will derive an upper bound for $N(\Gamma,\varepsilon)-3$ from the following linear programming  problem:\par 
Maximize the linear form 
$$\sum_{k=1}^{K(\varepsilon)}f_k$$
over the real numbers, subject to the constraints
$$\sum_{j=1}^{k}jf_j\leq (L+\delta)^d\left(\frac{k}{\varepsilon}\right)^d+k-1,$$ 
and $f_k\geq 0,$ for all $1\leq k\leq K(\varepsilon).$ We claim that $\mathbf{f}^{*}=\left(f_1^{*},\ldots,f_{K(\varepsilon)}^{*}\right)^T$ defined by
$$f_k^*:=\frac{k^d-(k-1)^d}{k}\left(\frac{L+\delta}{\varepsilon}\right)^d+\frac{1}{k},\quad \text{for each}\quad  2\leq k\leq K(\varepsilon),$$ 
$$f_1^*:=\left(\frac{L+\delta}{\varepsilon}\right)^d$$
is a feasible solution. Indeed, first notice that non-negativity is trivially satisfied as $d>0.$ For each $k,$ we also have a telescoping sum saturating the corresponding upper bound: 
\begin{align*} \sum_{j=1}^{k}jf_j^{*}&=\left(\frac{L+\delta}{\varepsilon}\right)^d+\sum_{j=2}^{k}j\left(\frac{j^d-(j-1)^d}{j}\left(\frac{L+\delta}{\varepsilon}\right)^d+\frac{1}{j}\right)\\
    &=\left(\frac{L+\delta}{\varepsilon}\right)^d+k-1+\sum_{j=2}^{k} \left(j^d-(j-1)^d\right)\left(\frac{L+\delta}{\varepsilon}\right)^d\\
    &=\left(\frac{L+\delta}{\varepsilon}\right)^dk^d+k-1.
\end{align*}
To use duality, we will first phrase the above linear programming problem in the notation of Definition \ref{thm:dual}. For this purpose, define a $K(\varepsilon)\times K(\varepsilon)$ matrix
\[
\mathbf{A}=
\begin{bmatrix} 
    1 & 0 & \dots & 0 & 0\\
    1 & 2 & \dots & 0 & 0\\
    \vdots & \vdots & \ddots & \vdots & \vdots\\ 
    1 & 2 & \dots & K(\varepsilon)-1 & 0\\
    1 & 2 & \dots & K(\varepsilon)-1 & K(\varepsilon)\\
\end{bmatrix},
\]
as well as the $K(\varepsilon)$-dimensional vectors  $\mathbf{f}=\left(f_1,\ldots,f_{K(\varepsilon)}\right)^T$, $\mathbf{c}=\left(1,\ldots,1\right)^T,$ and
\[
\mathbf{b}=
\begin{bmatrix}
    \left(\frac{L+\delta}{\varepsilon}\right)^d 1^d+0 \\
    \left(\frac{L+\delta}{\varepsilon}\right)^d 2^d+1\\
    \left(\frac{L+\delta}{\varepsilon}\right)^d 3^d+2\\
    \vdots \\
    \left(\frac{L+\delta}{\varepsilon}\right)^d(K(\varepsilon)-1)^d+K(\varepsilon)-2\\
    \left(\frac{L+\delta}{\varepsilon}\right)^dK(\varepsilon)^d+K(\varepsilon)-1\\
\end{bmatrix}.
\]
Then, in matrix form our problem is to maximize the linear form $\mathbf{c}^T\mathbf{f}$ among all $\mathbf{f}\in\mathbb{R}^{K(\varepsilon)},$
subject to the constraints
$\mathbf{Af}\leq\mathbf{b},$ and $\mathbf{f}\geq \mathbf{0}.$
From this, we may easily determine the form of the dual problem using Lines \ref{eqn: min1} and \ref{eq:min4}:
\par Minimize the linear form 
    $$\sum_{k=1}^{K(\varepsilon)}\left(\left(\frac{L+\delta}{\varepsilon}\right)^dk^d +k-1\right)g_k$$
among all $\mathbf{g}=\left(g_1,\ldots,g_{K(\varepsilon)}\right)^T\in\mathbb{R}^{K(\varepsilon)},$
subject to the conditions that for each fixed $1\leq j\leq K(\varepsilon),$
\begin{align}\label{eq:data}
\sum_{k=j}^{K(\varepsilon)}jg_k\geq 1, 
\end{align}
and $g_k\geq 0,$
for all $1\leq k \leq K(\varepsilon).$ By replacing the inequalities in Line \ref{eq:data} with equalities, we may recursively extract a non-negative sequence $g_k$ that is well-behaved for the dual problem. We demonstrate this as follows. Let 
$\textbf{g}^{*}=\left(g_1^{*},\ldots,g_{K(\varepsilon)}^{*}\right)^T.$ Beginning with $j=K(\varepsilon)$ in Line \ref{eq:data}, we set
$g_{K(\varepsilon)}^{*}={1}/{K(\varepsilon)}.$ Decreasing $j$ by 1 at each step; set
$g_k^{*}={1}/{k}-{1}/({k+1}),$ for $1\leq k\leq K(\varepsilon)-1.$ By the nature of this construction, $\mathbf{g}^*$ is feasible. We shall now establish that the conditions of Strong Duality Theorem \ref{thm:strongdual} apply to the feasible solutions $\mathbf{f}^*$ and $\mathbf{g}^*.$ 
 Observe that if 
 \begin{align*}
  a_k&:=\frac{1}{k},\quad \text{for }k\geq 1,\\
  b_k&:=\left(\frac{L+\delta}{\varepsilon}\right)^dk^d+k-1,\quad \text{for }k\geq 1,\quad b_0:=0,
 \end{align*}
 $s_k:=a_k(b_k-b_{k-1}),$ for $1\leq k\leq K(\varepsilon),$ $t_k:=b_k(a_k-a_{k+1})$ for $1\leq k\leq K(\varepsilon)-1,$ and $t_{K(\varepsilon)}:=a_{K(\varepsilon)}b_{K(\varepsilon)}$ as in Lemma \ref{combin}; then
 \begin{align*}
    f_k^*=s_k,\quad\text{and}\quad\left(\left(\frac{L+\delta}{\varepsilon}\right)^dk^d+k-1\right)g_k^*=t_k,
 \end{align*}
 for each $1\leq k\leq K(\varepsilon).$ Therefore, by Lemma \ref{combin} 
$$\sum_{k=1}^{K(\varepsilon)}f_k^*=\sum_{k=1}^{K(\varepsilon)}s_k=\sum_{k=1}^{K(\varepsilon)}t_k=\sum_{k=1}^{K(\varepsilon)}\left(\left(\frac{L+\delta}{\varepsilon}\right)^dk^d+k-1\right)g_k^*,$$
 verifying that $\mathbf{f}^{*},$ and $\mathbf{g}^{*}$ are optimal solutions for the primal problem and its dual, respectively, by Strong Duality Theorem \ref{thm:strongdual}. Since $\omega_{P(\varepsilon)}$ was an arbitrary subset of $\Gamma$ with separation $\delta(\omega_{P(\varepsilon)})\geq \varepsilon,$ we obtain
\begin{equation}\label{eq: logerror}
\begin{split}
    N(\Gamma,\varepsilon)-3\leq \sum_{k=1}^{K(\varepsilon)}f_{k}^{*}=\left(\frac{L+\delta}{\varepsilon}\right)^d \sum_{k=1}^{K(\varepsilon)}\frac{k^d-(k-1)^d}{k}+\sum_{k=2}^{K(\varepsilon)}\frac{1}{k}\\
\leq \left(\frac{L+\delta}{\varepsilon}\right)^d \sum_{k=1}^{\infty}\frac{k^d-(k-1)^d}{k}+\int_{1}^{K(\varepsilon)}\frac{1}{x}dx\\
=\left(\frac{L+\delta}{\varepsilon}\right)^dA_d+\log(K(\varepsilon)).
\end{split}
\end{equation} 
Because $K(\varepsilon)\to\infty,$ as $\varepsilon\to 0^+,$ there exists $\varepsilon_1=\varepsilon_1(\delta)>0$ such that for all $0<\varepsilon\leq\varepsilon_1,$ $\ell_j\leq (L+\delta)j^{-1/d}$ for all $j\geq K(\varepsilon)-1.$ Therefore, 
$$\varepsilon<\sum_{j=K(\varepsilon)-1}^{\infty}\ell_j\leq\sum_{j=K(\varepsilon)-1}^{\infty}\left(L+\delta\right)j^{-1/d}\leq(L+\delta)\int_{K(\varepsilon)-2}^{\infty}x^{-1/d}dx=\frac{d}{1-d}(L+\delta)(K(\varepsilon)-2)^{1-1/d}.$$
Thus 
$$K(\varepsilon)<\left(\frac{d}{1-d}\left(L+\delta\right)\right)^\frac{d}{1-d}\varepsilon^\frac{d}{d-1}+2=c_d\varepsilon^\frac{d}{d-1}+2.$$
Applying this to Inequality \ref{eq: logerror} we obtain 
\begin{align*}
\varepsilon^d(N(\Gamma,\varepsilon)-3)\leq(L+\delta)^dA_d+\varepsilon^d\log(K(\varepsilon))=(L+\delta)^dA_d+\varepsilon^d\log(c_d\varepsilon^{\frac{d}{d-1}}+2), 
\end{align*}
for all $0<\varepsilon\leq \min(\varepsilon_0,\varepsilon_1).$ Taking the limit supremum as $\varepsilon\to 0^+$ on each side, it follows that 
\begin{align*}
\limsup_{\varepsilon\to 0^+}N(\Gamma,\varepsilon)\varepsilon^d\leq(L+\delta)^dA_d+\limsup_{\varepsilon\to 0^+}\frac{d}{1-d}\varepsilon^d\log\left(\frac{1}{\varepsilon}\right)=(L+\delta)^dA_d,
\end{align*} where we apply L'H\^{o}spital's Rule to obtain the final equality. Finally, the result follows by sending $\delta\to 0^+.$
\end{proof}
\begin{theorem}[Lower Bound]\label{thm:lower}
    Fix $d\in(0,1),$ and $0<L<\infty.$ Assume $\Gamma=I\setminus \cup_{j=1}^{\infty}I_j$ is a cut-out set of Minkowski dimension $d$ satisfying Properties I and II, with $l_jj^{1/d}\sim L$ as $j\to\infty.$ Then 
$${L}^dA_d\leq \liminf_{\varepsilon\to0^{+}}N(\Gamma,\varepsilon)\varepsilon^d.$$
\end{theorem}
As mentioned earlier, the constructive lower bound that we provide below is the consequence of a greedy algorithm. Greedy algorithms do a good job of providing first-order asymptotic behavior for various discrete optimization problems (see e.g. Rolfes's work for covering spheres \cite{Rolfes} and several works concerning Riesz $s$-energy and polarization \cite{Bilyk_Mastrianni_Matzke_Steinerberger_2023}, \cite{López-García_McCleary_2022}). We stress that Property II (Monotonicty) of $\Gamma$ guarantees that this method works in the presence of Property I. Sharpness Theorem \ref{thm:sharp2} demonstrates that sets with a more complicated, nonmonotonic interval ordering, such as the self-similar fractals of Definition \ref{def:self-similar}, may have optimizers that do not resemble the asymptotically optimal configurations that we construct below.
\begin{proof}[Proof of Theorem \ref{thm:lower}]
    We make slight alterations to the arguments of Theorem \ref{thm:upper}.
    Fix $\delta>0,$ and take $u(\varepsilon)$ as in Lemma \ref{thm:Lemma}. Then, according to the lemma, there exists $\varepsilon_0=\varepsilon_0(\delta)>0$ such that for all $0<\varepsilon\leq \varepsilon_0,$ the inequality  $F(k,\varepsilon)\geq (L-\delta)^d\left({k}/{\varepsilon}\right)^d$
    holds for all $1\leq k\leq u(\varepsilon).
    $ 
    Define $$h_k:=\Big\lfloor\frac{k^d-(k-1)^d}{k}\left(\frac{L-\delta}{\varepsilon}\right)^d\Big\rfloor,$$ for $1\leq k\leq  v(\varepsilon),$ where  $v(\varepsilon):=\min\left(u(\varepsilon),w(\varepsilon)\right),$ $w(\varepsilon)=o(\varepsilon^{-d})$ and $w(\varepsilon)\to\infty$ as $\varepsilon\to 0^+,$ and $w(\varepsilon)$ is positive integer-valued.\footnote{Since $\frac{d}{1-d}>d$ for each $d\in(0,1),$ the growth restriction necessary for this part of the proof is stronger than the restriction placed on $u(\varepsilon)$ in hypothesis (ii) of Theorem \ref{thm:Lemma}} Then it follows from the definition of the floor function and a telescoping identity that for each $1\leq k\leq v(\varepsilon)$,
    \begin{equation}\label{Fbound}
    \begin{split}
          F(k,\varepsilon)+k-1\geq F(k,\varepsilon)\geq&(L-\delta)^d\left(\frac{k}{\varepsilon}\right)^d \\
      &=\sum_{j=1}^{k}j\left(\frac{j^d-(j-1)^d}{j}\right) \left(\frac{L-\delta}{\varepsilon}\right)^d\geq \sum_{j=1}^{k}jh_j.
    \end{split}
    \end{equation}
Using the above, we describe points of a large configuration $\omega_{R(\varepsilon)}\subset \Gamma$ with $\delta(\omega_{R(\varepsilon)})\geq\varepsilon.$ To begin, we assume without loss of generality that $\min\Gamma=0$. Then let $H(0)=1,$ and for each $1\leq k\leq v(\varepsilon),$
set $H(k)=\sum_{j=1}^{k}jh_j.$
Further, define
$\omega_{R(\varepsilon)}:=\bigcup_{k=0}^{v(\varepsilon)}\omega_{h_k},$
where 
$$\omega_{h_k}:=\left\{\sum_{j=1}^{H(k-1)+nk}l_j: 1\leq n \leq h_k \right\}\subset\Gamma,$$
 for $1\leq k\leq v(\varepsilon),$ $h_0:=0,$ and $\omega_{0}:=\{0\}.$
Then observe that for each $1\leq k\leq v(\varepsilon),$ $\delta(\omega_{h_k})\geq \varepsilon,$ by Inequality \ref{Fbound}. 
Since $\#\omega_{h_k}=h_k,$ and the configurations are pairwise disjoint, we have
$R(\varepsilon)=\#\omega_{R(\varepsilon)}=1+\sum_{k=1}^{v(\varepsilon)}h_k.$ To finish the proof, we first compute 
$$\sum_{k=1}^{v(\varepsilon)}h_k\geq \sum_{k=1}^{v(\varepsilon)}\left(\frac{k^d-(k-1)^d}{k}\left(\frac{L-\delta}{\varepsilon}\right)^d-1\right)\\
\geq\left(\sum_{k=1}^{v(\varepsilon)}\frac{k^d-(k-1)^d}{k}\left(\frac{L-\delta}{\varepsilon}\right)^d\right)-v(\varepsilon).$$
Then we immediately obtain by maximality and the definition of $v(\varepsilon)$ that
$$N(\Gamma,\varepsilon)\geq 1+\sum_{k=1}^{v(\varepsilon)}h_k\geq \left(\sum_{k=1}^{v(\varepsilon)}\frac{k^d-(k-1)^d}{k}\left(\frac{L-\delta}{\varepsilon}\right)^d\right)+o(\varepsilon^{-d})\quad \text{ as }\quad\varepsilon\to 0^+.$$ 
Finally, we may rearrange the above into
$$N(\Gamma,\varepsilon)\varepsilon^d\geq \left(\sum_{k=1}^{v(\varepsilon)}\frac{k^d-(k-1)^d}{k}(L-\delta)^d\right)+o(1)\quad\text{ as }\quad \varepsilon\to 0^+,$$ 
and the result follows by first sending $\varepsilon\to 0^+,$ then $\delta\to0^{+}.$
\end{proof}
We now move on to the proof of Proposition \ref{thm:limiting}, which establishes the limiting behavior of $p_d.$ As mentioned earlier, we shall need an integral representation of the Digamma function.
\begin{definition}[Digamma Function] The \textit{Digamma function} $\psi(z)$ is defined as the logarithmic derivative of the Gamma function:
$$\psi(z)=\frac{d}{dz}{\log(\Gamma(z))}=\frac{\Gamma'(z)}{\Gamma(z)}, \quad \text{for Re}(z)>0,$$
where 
$$\Gamma(z)=\int_{0}^{\infty}t^{z-1}e^{-t}dt,\quad \text{for Re}(z)>0.$$
We will use the following integral representation of the Digamma function that is originally due to Gauss (see Whittaker and Watson for one proof of this fact \cite{Whittaker_Watson_1920b}):
\begin{equation}\label{digamma}
\psi(z)=\int_{0}^{\infty}\left(\frac{e^{-t}}{t}-\frac{e^{-zt}}{1-e^{-t}}\right)dt. 
\end{equation}
\end{definition}
\begin{proof}[Proof of Proposition \ref{thm:limiting}] Notice that the function 
$f(s):=\sum_{n=1}^{\infty}(n^s-(n-1)^s)/{n}$ providing the packing constant in Theorem \ref{thm:main} may be rewritten as a reflection of a Dirichlet series 
$$f(s)=\sum_{n=1}^{\infty}n^s\left(\frac{1}{n}-\frac{1}{n+1}\right)=\sum_{n=1}^{\infty}\frac{1}{(n+1)n^{1-s}},$$ using the partial summation formula provided in Lemma \ref{combin}. Indeed, if $g(s)=f(1-s),$ then since $f(s)$ converges for real $s<1,$ and diverges when $s>1,$ it follows from the convergence theory of Dirichlet series (see \cite{Hardy_Riesz_1915}) that the \textit{absissca of convergence} of the Dirichlet series $g(s)$ is $\sigma=0,$ so the series converges absolutely for all complex $s\in\mathbb{C}$ with $\text{Re}(s)>0.$ Next, notice that for real $s>0,$ and $n\geq 1$ the inequalities
$$\frac{1}{(n+2)(n+1)^s}\leq \int_{n}^{n+1}\frac{1}{(x+1)x^s}dx\leq \frac{1}{(n+1)n^s}$$
provide
$$g(s)-\frac{1}{3 \cdot 2^s}\leq \int_{1}^{\infty}\frac{1}{(x+1)x^s}dx\leq g(s).$$
The interior integral above may be represented as a difference of Digamma functions using the substitution $x=e^{t/2}$ and Equation \ref{digamma}:
\begin{align*}
    \frac{1}{2}\left(\psi\left(\frac{s+1}{2}\right)-\psi\left(\frac{s}{2}\right)\right)&=\frac{1}{2}\int_{0}^{\infty}\frac{e^{-\frac{s+1}{2}t}-e^{-\frac{s}{2}t}}{1-e^{-t}}dt\\
    &=\int_{1}^{\infty}\frac{1}{(x+1)x^s}dx.
\end{align*}
Noting that $2^{1-d}\to 1,$ as $d\to1^-,$ we have
\begin{align*}
    \lim_{d\to 1^{-}}p_d=\lim_{s\to 1^{-}}(1-s)f(s)=\lim_{s\to 0^+}sg(s)&=\lim_{s\to 0^+}s\int_{1}^{\infty}\frac{1}{(x+1)x^s}dx\\
    &=\lim_{s\to 0^+}\frac{s}{2}\left(\psi\left(\frac{s+1}{2}\right)-\psi\left(\frac{s}{2}\right)\right)\\
    &=-\lim_{s\to 0^+}s\psi(s)=1,
\end{align*}
because the only singularities of $\psi(s)$ are simple poles at each non-positive integer, and each has residue -1.
\end{proof}
\begin{remark}
The above shows that $f(s)$ possesses a simple pole of residue $1$ at $s=1.$
\end{remark}
\begin{proof}[Proof of Proposition \ref{thm:inv}] Clearly, Line \ref{eq:p2} follows from Line \ref{eq:p1} so we proceed to prove the inequalities in Line \ref{eq:p1}. This proof uses some of the ideas expressed in Theorem 4.1 of \cite{Lapidus1993}. Let $\varepsilon>0,$ and observe that the definition of $F(k,\varepsilon)$ preceding Lemma \ref{thm:Lemma} depends only on the sequence $(l_j)_{j=1}^{\infty}.$ Thus, if $\delta>0,$ we may find $\varepsilon_0=\varepsilon_0(\delta)>0$ such that $F(k,\varepsilon)\left(\varepsilon/k\right)^d\geq \left(\underline{L}-\delta\right)^d,$ for all $1\leq k\leq u(\varepsilon),$ just as in Theorem \ref{thm:lower}. Following the remaining arguments of Theorem \ref{thm:lower}, let $v(\varepsilon)$ be as defined there, then define 
$$h_k:=\left\lfloor\frac{k^d-(k-1)^d}{k}\left(\frac{\underline{L}-\delta}{\varepsilon}\right)^d\right\rfloor,$$
$H(k)=\sum_{j=1}^{k}jh_j,$ for $1\leq k\leq v(\varepsilon),$ and observe that the subsets
$$\gamma_{n,h_k}:=\left\{H(k-1)+(n-1)k+1,\ldots,H(k-1)+nk\right\}$$
obey the mass requirements 
$$\sum_{j\in\gamma_{n,h_k}}l_j\geq \varepsilon,$$
for all $1\leq n\leq h_k,$ and all
$1\leq k\leq v(\varepsilon).$ 
Letting $$\gamma'_{v(\varepsilon),h_{v(\varepsilon)}}:=\gamma_{v(\varepsilon),h_{v(\varepsilon)}}\bigcup\{n\in\mathbb{N}\text{ : }n\geq H(v(\varepsilon))+1\},$$ provides a partition of $\mathbb{N}$ containing $H(v(\varepsilon))$ subsets obeying the mass requirements. Therefore,
$$N(\mathcal{L},\varepsilon)\varepsilon^d\geq \varepsilon^d H(v(\varepsilon))\geq \left(\sum_{k=1}^{v(\varepsilon)}\frac{k^d-(k-1)^d}{k}\left(\underline{L}-\delta\right)^d\right)+o(1)\quad\text{ as }\quad \varepsilon\to 0^+,$$
and the result follows by first sending $\varepsilon\to 0^+,$ then $\delta\to 0^+.$
\par Moving on to the upper bound, we first claim that for any maximal partition $\omega_{N(\mathcal{L},\varepsilon)}:=\{\gamma_i:1\leq i\leq N(\mathcal{L},\varepsilon)\}$ realizing the mass requirements 
$$\sum_{j\in\gamma_i}l_j\geq \varepsilon,$$ one may assume without loss of generality that
\begin{align}\label{eq:wlog}
    \{j\}\in \omega_{N(\mathcal{L},\varepsilon)}, \quad \text{ for all } 1\leq j\leq F(1,\varepsilon).
\end{align} 
Indeed, if any such singleton $\{j\}$ does not belong to $\omega_{N(\mathcal{L},\varepsilon)},$ we can find $\gamma_{n_j}$ containing $j,$ then
define 
$\gamma'_{n_j}=\{j\}$ and redistribute the mass of $\gamma_{n_j}$ to any of the other remaining subsets of $\omega_{N(\mathcal{L},\varepsilon)}.$ This process relocates the mass of each such $l_j$ to a singleton without decreasing the number of sets in $\omega_{N(\mathcal{L},\varepsilon)},$ or violating the mass requirements. In the sequel, we will thus work under the assumption that Line \ref{eq:wlog} is satisfied. From this assumption, notice that 
$$(N(\mathcal{L},\varepsilon)- F(1,\varepsilon))\varepsilon\leq \sum_{i=F(1,\varepsilon)+1}^{N(\mathcal{L},\varepsilon)}\sum_{j\in\gamma_i}l_j=\sum_{j=F(1,\varepsilon)+1}^{\infty}l_j,$$ 
which implies 
\begin{align}\label{eq:tails}
    N(\mathcal{L},\varepsilon)\leq F(1,\varepsilon)+\frac{1}{\varepsilon}\sum_{j=F(1,\varepsilon)+1}^{\infty}l_j
\end{align}
Because $l_j$ is $d$-dimensional, for a fixed $\delta>0$ we may find $M(\delta)\in\mathbb{N}$ such that for all $j\geq M,$ $(\underline{L}-\delta)j^{-1/d}\leq l_j\leq (\overline{L}+\delta)j^{-1/d}$.
Thus, if $\varepsilon_0=\varepsilon_0(\delta)>0$ is chosen small enough so that $F(1,\varepsilon_0)\geq M,$ one may apply an integral estimate to the tail written in Line \ref{eq:tails} to obtain
\begin{align*}
\frac{1}{\varepsilon}\sum_{j=
F(1,\varepsilon)+1}^{\infty}l_j\leq \left(\frac{\overline{L}+\delta}{\varepsilon}\right)\sum_{j=
F(1,\varepsilon)+1}^{\infty}j^{-1/d}
&\leq \left(\frac{\overline{L}+\delta}{\varepsilon}\right)\int_{F(1,\varepsilon)}^{\infty}x^{-1/d}dx\\&=\left(\frac{\overline{L}+\delta}{\varepsilon}\right)\frac{F(1,\varepsilon)^{1-1/d}}{\frac{1}{d}-1},
\end{align*}
for all $\varepsilon\leq\varepsilon_0.$ Plugging the above into Line \ref{eq:tails}, multiplying through by $\varepsilon^d,$ and using the fact that $F(1,\varepsilon)\geq M$ for all $\varepsilon\leq\varepsilon_0(\delta),$ we see that 
\begin{align*}
    N(\mathcal{L},\varepsilon)\varepsilon^d&\leq  F(1,\varepsilon)\varepsilon^d+\frac{\overline{L}+\delta}{\frac{1}{d}-1}\left(F(1,\varepsilon)\varepsilon^d\right)^{d/(1-d)}\\
    &\leq (\overline{L}+\delta)^d+\frac{d}{1-{d}}(\overline{L}+\delta)(\underline{L}-\delta)^{d-1},
\end{align*}
for all $\varepsilon\leq\varepsilon_0(\delta).$ Therefore, the result follows by first sending $\varepsilon\to 0^+,$ then $\delta\to 0^+.$
\end{proof}
\end{subsection}
\begin{subsection}{Proof of Sharpness Theorem \ref{thm:sharp1}}\label{sec:sharp1}
Next, we discuss why Property I (Minkowski measurability) is necessary for convergence of the first-order packing asymptotics on some cut-out sets. We will prove a sufficient result (Theorem \ref{thm:exact}) that gives the best packing asymptotics along a continuum of subsequences for the example set $S$ of Example \ref{ternary}.
\par To state Theorem \ref{thm:exact} below, we will first need some additional notation. Suppose $S\subset[0,1]$ is the monotonically rearranged 1/3 Cantor set described in Example \ref{ternary}. Let $a\in(1,3]$ be any fixed real number and define a sequence $\varepsilon_{a,n}=a 3^{-n},$ $n\in \mathbb{N},$ $n\geq 1.$ Then for each $k\geq 0$ let
\begin{align*}
p_k:=\lceil a3^{k-1}\rceil.
\end{align*}
Note that $p_0=1$  by the definition of the ceiling function, as $a\in(1,3].$
\begin{theorem}\label{thm:exact} With notation listed as above, let $s=\log2/\log3\in(0,1)$ represent the  Minkowski dimension of $S\subset[0,1].$ 
Then for each $a\in(1,3],$ 
\begin{align}\label{eq:cantorlims}
L(a):=\lim_{n\to\infty}N\left(S,\varepsilon_{a,n}\right)\varepsilon_{a,n}^s=2^{\log_3 a-1}\left(1+\sum_{k=1}^{\infty}\frac{2^{k-1}}{p_k}\right). 
\end{align}
In particular,
\begin{align}\label{eq:cantub}
  L(a)\leq 2^{\log_3 a-1}\left(1+3^{1-\log_3 a}\right), \quad \text{ for all $a\in(1,3],$} \quad \text{ and }\quad L(3)=2.
\end{align}
\end{theorem}
That the claims in Line \ref{eq:cantub} follow from Equation \ref{eq:cantorlims} is immediate from the definition of the ceiling function and the formula for a convergent geometric series. Next, notice that we may optimize the upper bound for $L(a)$ as a function of $\log_3 a.$ Indeed, let $f(x)=2^{x-1}\left(1+3^{1-x}\right)$ such that $L(a)\leq f(\log_3 a)$ for all $a\in(1,3].$ Then $f(x)$ has a unique minimum and maximum over $x\in(0,1]$ at 
$$\underline{x}=\frac{\log\left(\frac{3\log(\frac{3}{2})}{\log2}\right)}{\log3}$$
and $\overline{x}=1,$ respectively. Thus, the values of $\underline{a}=3^{\underline{x}}=3\log(3/2)/\log(2)$ and $\overline{a}=3^{\overline{x}}=3$ provide the optimizing parameters, and we have   
$$\inf_{a\in(1,3]}L(a)\leq L(\underline{a})\leq f(\underline{x})=\frac{\log3}{\log(\frac{3}{2})}\cdot 2^{{\log\left(\frac{3\log(\frac{3}{2})}{\log2}\right)}/{\log3}}\approx 1.93, \quad \text{ while }\max_{a\in(1,3]}L(a)=L(3)=2.$$
These facts imply Sharpness Theorem \ref{thm:sharp1}. Thus, we proceed with the proof of Theorem \ref{thm:exact} below.
\begin{proof}[Proof of Theorem \ref{thm:exact}]
In the proof below, we adapt the basic techniques from the measurable situation to prove convergence of optimal packing along certain subsequences in the non-measurable set $S$. The main idea is to consider the behavior of the distribution function $F(k,\varepsilon_{a,n})$ along the subsequences $\varepsilon_{a,n}=a 3^{-n}$. 
\par We argue that for each $k\geq 0$ and each $p_k\leq p<p_{k+1},$ there holds 
\begin{equation}\label{eq:pkbounds}
    2^{n+k-1}-p_k\leq F(p,a3^{-n})\leq 2^{n+k-1}-1.
\end{equation}
Proceeding with the proof of Line \ref{eq:pkbounds}, let $k\geq 0$ and $p_k\leq p<p_{k+1}.$ Then by the monotonicity of $F(p,\varepsilon_{a,n})$ in its first argument, it suffices to show that 
$$F(p_{k},a3^{-n})\geq 2^{n+k-1}-p_k,$$
and
$$F(p_{k+1}-1,a3^{-n})\leq 2^{n+k-1}-1.$$
To this end, observe that by the definition of the sequence $l_j$ corresponding to $S,$ and the definition of $p_k$ 
$$\sum_{j=2^{n+k-1}-p_k}^{2^{n+k-1}-1}l_j\geq p_k 3^{-(n+k-1)} \geq a3^{k-1}3^{-(n+k-1)}= a3^{-n}.$$ Thus, the lower bound of Line \ref{eq:pkbounds} is proven. Next, using the minimality of $p_{k+1},$ we have that 
$$\sum_{j=2^{n+k-1}}^{2^{n+k-1}+p_{k+1}-2}l_j\leq 3^{-(n+k)}\left(p_{k+1}-1\right)<a3^{-n}.$$
Therefore, the upper bound of Line \ref{eq:pkbounds} holds, and the estimate is true for all $p$ in the range. 
\par We now turn to the objective of \textit{using} the bounds provided in Line \ref{eq:pkbounds} to control the set of adjacent distances in a packing of radius at least $\varepsilon_{a,n}$. To this end, suppose that $\omega_{P(\varepsilon_{a,n})}\subset S$ is a collection of points such that $\delta(\omega_{P(\varepsilon_{a,n})})\geq \varepsilon_{a,n}.$ Just as in Theorem \ref{thm:upper}, order the elements of $\omega_{P(\varepsilon_{a,n})}$ by
$0\leq x_1<x_2<\cdots<x_{P(\varepsilon_{a,n})}\leq1.$ Similarly, there is the sequence of adjacent distances $x_{n+1}-x_n=p(j_n+1,k_n),$ where $x_n=y_{j_n},$ ($j_n$ provided by Proposition \ref{thm:simpchar}). Thus, we may study the frequencies $f(k,\omega_{P(\varepsilon_{a,n})})$ defined as earlier on, for $k\geq 1.$ Let
$$I(\varepsilon_{a,n}):=\min\left\{i\geq 0\text{ : } \sum_{j=p_{i}}^{\infty}l_j\leq \varepsilon_{a,n}\right\},$$ and notice that $I(\varepsilon_{a,n})\to\infty$ as $n\to\infty$ by  the definition of $p_k$ and the fact that $\varepsilon_{a,n}\to 0$ as $n\to\infty.$ From the consideration of the largest point of $\omega_{P(\varepsilon_{a,n})},$ there holds
$$\#\left(\{\omega_{P(\varepsilon_{a,n})}\}\cap \left(\{y_{j_n}\text{ : } j_n\geq p_{I(\varepsilon_{a,n})}\}\cup \{1\}\right)\right)\leq 2;$$ thus, it follows that 
$$P(\varepsilon_{a,n})-3\leq \sum_{p=1}^{p_{I(\varepsilon_{a,n})}}f(p,\omega_{P(\varepsilon_{a,n})})$$
by the same reasoning of Theorem \ref{thm:upper}. The proof of the asymptotic upper bound proceeds as follows: Using the upper bound of Line \ref{eq:pkbounds}, and recycling the reasoning of Theorem \ref{thm:upper}, one obtains that for each $k\geq0$ and each $p_k\leq p<p_{k+1},$
$$\sum_{j=1}^{p}jf(j,\omega_{P(\varepsilon_{a,n})})\leq F(p,a3^{-n})+p-1\leq 2^{n+k-1}+p-1.$$ 
Similarly to our prior arguments, for each $1\leq p\leq p_{I(\varepsilon_{a,n})},$ let $f_p:=f(p,\omega_{P(\varepsilon_{a,n})}).$ We provide an upper bound for $N(S,\varepsilon_{a,n})-3$ by solving the following linear programming  problem:
\par
Maximize the linear form
$$\sum_{p=1}^{p_{I(\varepsilon_{a,n})}}f_p$$
over the non-negative reals subject to the constraints that 
\begin{equation}\label{eq:constraints}
\begin{split}
\sum_{j=1}^{p}jf_j&\leq {2^{n+k-1}+p-1}, \quad \text{ for each $p_k\leq p<p_{k+1},$ where $0\leq k\leq I(\varepsilon_{a,n})-1$},\text{ and }\\
\quad \sum_{j=1}^{p_{I(\varepsilon_{a,n})}}jf_j&\leq 2^{n+I(\varepsilon_{a,n})-1}+p_{I(\varepsilon_{a,n})}-1
\end{split}
\end{equation}
The analogies with the proof of Theorem \ref{thm:upper} should contribute to the reader's intuition for the following: We claim that 
$
\mathbf{f^*}:=\left(f^*_1,\ldots,f^*_{p_{I(\varepsilon_{a,n})}}\right)$ defined by 
\begin{align*}
f^*_p:=
    \begin{dcases}
        2^{n-1}, & \text{ if } p=1\\
        \frac{2^{n+k-2}+p_{k}-p_{k-1}}{p_{k}}, & \text{ if } p=p_{k}, \text{ for some } 1\leq k\leq I(\varepsilon_{a,n})\\
        0, & \text{otherwise}\\
    \end{dcases}
\end{align*}
is optimal for the aforementioned linear programming  problem. Notice that $1f_1^*=f_1^*=2^{n-1},$ and for each $k\geq 1$
\begin{align*}
    \sum_{j=1}^{p_k}jf_j^*&=\sum_{m=0}^{k}p_mf_{p_m}^*=2^{n-1}+\sum_{m=1}^{k}p_m\frac{2^{n+m-2}+p_m-p_{m-1}}{p_m}\\
    &=\sum_{m=1}^{k}(2^{n+m-2}+p_m-p_{m-1})\\
    &=2^{n-1}+\sum_{m=1}^{k}2^{n+m-2}+\sum_{m=1}^{k}(p_m-p_{m-1})\\
    &=2^{n-1}+2^{n-1}\left(2^{k}-1\right)+p_k-p_0=2^{n+k-1}+p_k-1.
\end{align*}
Thus the constraints in Line \ref{eq:constraints} are satisfied and $\mathbf{f^*}$ is feasible. Define a pair of infinite sequences $(a_k)_{k=1}^{\infty},$ and $(b_k)_{k=0}^{\infty}$ by
\begin{align}
\label{eq:ak} a_k&:=\frac{1}{p_{k-1}}, \quad \text{ for all } k\geq 1,\\
\label{eq:bk} b_k&:=
    \begin{dcases}
        0, & \text{ if } k=0\\
        2^{n+k-2}+p_{k-1}-1, & \text{ if } k\geq 1\\
    \end{dcases}    
\end{align}
Then, if $s_k:=a_k\left(b_k-b_{k-1}\right),$ for $1\leq k\leq I(\varepsilon_{a,n})+1$ as in Lemma \ref{combin}, we find
\begin{align*}\sum_{k=1}^{I(\varepsilon_{a,n})+1}s_k&=\sum_{k=1}^{I(\varepsilon_{a,n})+1}a_k(b_k-b_{k-1})\\&=a_1b_1+\sum_{k=2}^{I(\varepsilon_{a,n})+1}\frac{1}{p_{k-1}}\left(2^{n+k-2}+p_{k-1}-1-\left(2^{n+k-3}+p_{k-2}-1\right)\right)\\
&=2^{n-1}+\sum_{k=2}^{I(\varepsilon_{a,n})+1}\frac{2^{n+k-3}+p_{k-1}-p_{k-2}}{p_{k-1}}=2^{n-1}+\sum_{k=1}^{I(\varepsilon_{a,n})}\frac{2^{n+k-2}+p_{k}-p_{k-1}}{p_k}\\&=\sum_{k=0}^{I(\varepsilon_{a,n})}f_{p_k}^*=\sum_{p=1}^{{p_{I(\varepsilon_{a,n})}}}f_p^*.
\end{align*}
Again, we seek to apply partial summation Lemma \ref{combin} along with Strong Duality Theorem \ref{thm:dual}. The dual problem to the constrained maximization of $\sum_{p=1}^{p_{I(\varepsilon_{a,n})}}f_p$ is to minimize the linear form 
$$\left(2^{n+I(\varepsilon_{a,n})-1}+p_{I(\varepsilon_{a,n})}-1\right)g_{p_{I(\varepsilon_{a,n})}}+\sum_{k=0}^{{I(\varepsilon_{a,n})-1}}\sum_{p=p_k}^{p_{k+1}-1}(2^{n+k-1}+p-1) g_p,$$
among all $\mathbf{g}=\left(g_1,\ldots,g_{p_{I(\varepsilon_{a,n})}}\right)^T$ with non-negative real entries, subject to the linear constraints requiring that for each $1\leq m\leq p_{I(\varepsilon_{a,n})},$
\begin{align*}
    \sum_{p=m}^{p_{I(\varepsilon_{a,n})}}m g_p\geq 1.
\end{align*}
Using the same procedure following Line \ref{eq:data} of Theorem \ref{thm:upper}, we construct the feasible dual program $\mathbf{g^*}:=\left(g_1^*,\ldots,g_{p_{I(\varepsilon_{a,n})}}^*\right)^T$ with entries
\begin{align*}
g^*_p:=
    \begin{dcases}
        \frac{1}{p_{I(\varepsilon_{a,n})}}, & \text{ if } p=p_{I(\varepsilon_{a,n})}\\
        \frac{1}{p_{k-1}}-\frac{1}{p_{k}}, & \text{ if } p=p_{k-1}, \text{ for some } 1\leq k\leq I(\varepsilon_{a,n})\\
        0, & \text{otherwise}.\\
    \end{dcases}
\end{align*}
Let $t_k:=b_k(a_k-a_{k+1}),$ for $1\leq k\leq I(\varepsilon_{a,n}),$ $t_{I(\varepsilon_{a,n})+1}:=a_{I(\varepsilon_{a,n})+1}b_{I(\varepsilon_{a,n})+1}$ and suppose $a_k,b_k$ are given as in Lines \ref{eq:ak} and \ref{eq:bk}, respectively. Then
\begin{align*}
&\left(2^{n+I(\varepsilon_{a,n})-1}+p_{I(\varepsilon_{a,n})}-1\right)g_{p_{I(\varepsilon_{a,n})}}^*+\sum_{k=0}^{{I(\varepsilon_{a,n})-1}}\sum_{p=p_k}^{p_{k+1}-1}(2^{n+k-1}+p-1) g_p
^*\\
&=\left(2^{n+I(\varepsilon_{a,n})-1}+p_{I(\varepsilon_{a,n})}-1\right)g_{p_{I(\varepsilon_{a,n})}}^*+\sum_{k=1}^{I(\varepsilon_{a,n})}(2^{n+k-2}+p_{k-1}-1)g_{p_{k-1}}^*\\ 
&=\left(2^{n+I(\varepsilon_{a,n})-1}+p_{I(\varepsilon_{a,n})}-1\right)\frac{1}{p_{I(\varepsilon_{a,n})}}+\sum_{k=1}^{I(\varepsilon_{a,n})}(2^{n+k-2}+p_{k-1}-1)\left(\frac{1}{p_{k-1}}-\frac{1}{p_k}\right)\\
&=a_{I(\varepsilon_{a,n})+1}b_{I(\varepsilon_{a,n})+1}+\sum_{k=1}^{I(\varepsilon_{a,n})}b_k(a_k-a_{k+1})=\sum_{k=1}^{I(\varepsilon_{a,n})+1}t_k,
\end{align*}
Thus, by Lemma \ref{combin}, we have
\begin{align*}
\sum_{p=1}^{{p_{I(\varepsilon_{a,n})}}}f_p^*=\sum_{k=1}^{I(\varepsilon_{a,n})+1}s_k&=\sum_{k=1}^{I(\varepsilon_{a,n})+1}t_k\\
&=\left(2^{n+I(\varepsilon_{a,n})-1}+p_{I(\varepsilon_{a,n})}-1\right)g_{p_{I(\varepsilon_{a,n})}}^*+\sum_{k=0}^{{I(\varepsilon_{a,n})-1}}\sum_{p=p_k}^{p_{k+1}-1}(2^{n+k-1}+p-1) g_p^*.
\end{align*}
Therefore $\mathbf{f^*}$ and $\mathbf{g^*}$ are optimal for the primal problem and its dual, by Theorem \ref{thm:dual}. Thus,
\begin{align*}
  N(S,\varepsilon_{a,n})-3\leq \sum_{p=1}^{p_{I(\varepsilon_{a,n})}}f_p^*&=2^{n-1}+\sum_{k=1}^{I(\varepsilon_{a,n})}\frac{2^{n+k-2}+p_{k}-p_{k-1}}{p_k}\\&=2^{n-1}\left(1+\sum_{k=1}^{I(\varepsilon_{a,n})}\frac{2^{k-1}}{p_k}\right)+\sum_{k=1}^{I(\varepsilon_{a,n})}\frac{p_{k}-p_{k-1}}{p_k}\\
  &\leq 2^{n-1}\left(1+\sum_{k=1}^{\infty}\frac{2^{k-1}}{p_k}\right)+\sum_{k=1}^{I(\varepsilon_{a,n})}\frac{\lceil a3^{{k-1}}\rceil-\lceil a 3^{k-2}\rceil }{\lceil  a 3^{k-1}\rceil}&\\
  &\leq 2^{n-1}\left(1+\sum_{k=1}^{\infty}\frac{2^{k-1}}{p_k}\right)+\sum_{k=1}^{I(\varepsilon_{a,n})}\frac{a3^{{k-1}}+1-a 3^{k-2} }{a 3^{k-1}}\\
  &\leq 2^{n-1}\left(1+\sum_{k=1}^{\infty}\frac{2^{k-1}}{p_k}\right)+\frac{1}{a}\sum_{k=1}^{\infty}\frac{1}{3^{k-1}}+\sum_{k=1}^{I(\varepsilon_{a,n})}\frac{2}{3}\\
  &=2^{n-1}\left(1+\sum_{k=1}^{\infty}\frac{2^{k-1}}{p_k}\right)+\frac{3}{2a}+\frac{2}{3}I(\varepsilon_{a,n}),
\end{align*}
where we used the formula for the sum of a geometric series with common ratio $r=1/3$, and the fact that $x\leq \lceil x\rceil\leq x+1 $ for any non-negative real number $x.$ By the definitions of $I(\varepsilon_{a,n}),$ and $\ell_j,$ there holds
\begin{align*}
a3^{-n}=\varepsilon_{a,n}<\sum_{j=p_{I(\varepsilon_{a,n})-1}}^{\infty}\ell_j=\sum_{j=\lceil a 3^{I(\varepsilon_{a,n})-2}\rceil}^{\infty}\ell_j\leq\sum_{3^{I(\varepsilon_{a,n})-3}}^{\infty}\ell_j\leq \sum_{j=2^{I^*}}^{\infty}\ell_j
&=\sum_{k=I^*}^{\infty}\sum_{j=2^{k}}^{2^{k+1}-1}3^{-k-1}\\
&=\sum_{k=I^*}^{\infty}2^k3^{-k-1}\\
&=\frac{1}{3}\cdot\frac{(2/3)^{I^*}}{1-\frac{2}{3}}=\left(\frac{2}{3}\right)^{I^*},
\end{align*}
where we have assigned $I^*:=\lfloor \frac{\log3}{\log2}\left(I(\varepsilon_{a,n})-3\right)\rfloor.$ By taking logarithms on the left and right hand-side of the above, we find that 
$$\log\left(\frac{3}{2}\right)I^{*}<n\log(3)-\log(a),$$
and hence
$$I(\varepsilon_{a,n})<n\frac{\log2}{\log(3/2)}+c_a,$$
where $c_a$ is a constant depending on the value of $a.$
Thus, 
$$\varepsilon_{a,n}^s\left(\frac{3}{2a}+\frac{2I(\varepsilon_{a,n})}{3}+c_a\right)<{3a^{\frac{\log2}{\log3}-1}}\cdot\frac{1}{2^{n+1}}+\frac{\log(2)a^\frac{\log2}{\log3}}{3\log(3/2)}\cdot \frac{n}{2^{n-1}}+\frac{c_a a^\frac{\log2}{\log3}}{3}\cdot\frac{1}{2^{n-1}}.$$
Limiting as $n\to\infty,$ we finally obtain
\begin{align*}
\limsup_{n\to\infty}N(S,\varepsilon_{a,n})\varepsilon_{a,n}^s&=\limsup_{n\to\infty}N(S,a3^{-n})(a3^{-n})^s\\
&\leq \limsup_{n\to\infty}(a3^{-n})^s2^{n-1}\left(1+\sum_{k=1}^{\infty}\frac{2^{k-1}}{p_k}\right)\\&+\limsup_{n\to\infty}\left({3a^{\frac{\log2}{\log3}-1}}\cdot\frac{1}{2^{n+1}}+\frac{\log(2)a^\frac{\log2}{\log3}}{3\log(3/2)}\cdot \frac{n}{2^{n-1}}+\frac{c_a a^\frac{\log2}{\log3}}{3}\cdot\frac{1}{2^{n-1}}\right)\\
&=(3^{\log_3a})^s 2^{-1}\left(1+\sum_{k=1}^{\infty}\frac{2^{k-1}}{p_k}\right)=2^{\log_3a-1}\left(1+\sum_{k=1}^{\infty}\frac{2^{k-1}}{p_k}\right),
\end{align*}
so that half of the theorem is proved. 
\par We shall now proceed to the lower bound. Similarly to the proof of Theorem \ref{thm:lower}, we accomplish this by explicitly constructing a packing that saturates the limit supremum as $n\to\infty.$ Quite similarly to the lower bound of Line \ref{eq:part1} in Theorem  \ref{thm:Lemma}, let $J(n)$ be an integer-valued sequence such that $J(n)\to\infty,$ and $J(n)=o(n),$
as $n\to\infty.$ Observe that for all $0\leq k \leq J(n),$ and all $p_k\leq p<p_{k+1},$
$$F(p,a3^{-n})\geq 2^{n+k-1}-p_k\geq 2^{n+k-1}\left(1-\left(\frac{3}{2}\right)^{k-1}\frac{3}{2^n}\right)=2^{n+k-1}(1+o(1)),\quad \text{ as $n\to\infty$}.$$ 
Thus, for a fixed $0<\delta<1,$ there exists $n_0=n_0(\delta)\in\mathbb{N}$ such that for all $n\geq n_0,$ $0\leq k\leq G(n):=\min\left(J(n),I(\varepsilon_{a,n})\right),$ and $p_k\leq p <p_{k+1},$ there holds
$$F(p,a3^{-n})\geq 2^{n+k-1}\left(1-\delta\right).$$
Define $h_p$ over the range $1\leq p\leq p_{G(n)}$ by 
\begin{align*}
h_p:=
    \begin{dcases}
        \lfloor (1-\delta)2^{n-1}\rfloor, & \text{ if } p=1\\
        \big\lfloor\frac{(1-\delta)2^{n+k-2}}{p_k}\big\rfloor, & \text{ if } p=p_{k}, \text{ for some } 1\leq k\leq G(n)\\
        0, & \text{otherwise}.\\
    \end{dcases}
\end{align*}
Then, a familiar telescoping sum gives
$$\sum_{j=1}^{p}jh_j\leq (1-\delta)\left(2^{n-1}+\sum_{k=1}^{G(n)}p_k \frac{2^{n+k-2}}{p_k}\right)\leq (1-\delta)2^{n+k-1}\leq F(p,a3^{-n})\leq F(p,a3^{-n})+p-1,$$
for each such $p.$ Next, for each $1\leq p \leq {p_{G(n)}},$ set 
$H(p):=\sum_{j=1}^{p}jh_j,$ 
and notice that 
\begin{align}\label{eq:H}  H({p_{G(n)}})=p_0h_{p_0}+p_1h_{p_1}+\cdots+p_{G(n)}h_{p_{G(n)}},
\end{align}
because $h_p=0,$ if $p\notin\{p_k: 0\leq k\leq G(n)\}.$
We are now well-positioned to construct greedy configurations with large separations over $S$. To this end, define
$$\omega_1:=\left\{\sum_{j=1}^{n+1}l_j: \text{ } 0\leq n\leq h_{0}-1\right\}\cup\{0\},$$
and
$$\omega_{h_{p_k}}:=\left\{\sum_{j=1}^{H(p_{k-1})+np_k}l_j:\text{ }1\leq n \leq h_{p_k}\right\}\subset S,$$
for each $1\leq k\leq G(n).$ Then if $$\omega_{R(\varepsilon_{a,n})}:=\bigcup_{k=0}^{G(n)}\omega_{h_{p_k}}$$
we obtain that $\delta(\omega_{R(\varepsilon_{a,n})})\geq \varepsilon_{a,n}$ by the construction, 
and $R(\varepsilon_{a,n})=\#\omega_{R(\varepsilon_{a,n})}=1+H(G(n)).$ Counting the non-zero contributions to Equation \ref{eq:H} and using the definition of $h_p,$ we obtain
\begin{align*}    H(G(n))\geq\sum_{p=1}^{p_{G(n)}}h_p&\geq(1-\delta)2^{n-1}-1+\sum_{k=1}^{G(n)}\left((1-\delta)\frac{2^{n+k-2}}{p_k}-1\right)\\
    &\geq(1-\delta)2^{n-1}\left(1+\sum_{k=1}^{G(n)}\frac{2^{k-1}}{p_k}\right)-\left(G(n)+1\right).
\end{align*}
Thus, by the maximality of $N(S,\varepsilon_{a,n}),$ and the growth restriction on $G(n),$ we have
$$
N(S,\varepsilon_{a,n})\geq(1-\delta)2^{n-1}\left(1+\sum_{k=1}^{G(n)}\frac{2^{k-1}}{p_k}\right)+o\left(n\right)\quad\text{ as }\quad n\to\infty, 
$$
from which it follows that
$$N(S,\varepsilon_{a,n})\varepsilon_{a,n}^s\geq(1-\delta)2^{\log_3a-1}\left(1+\sum_{k=1}^{G(n)}\frac{2^{k-1}}{p_k}\right)+o(1)\quad\text{ as }\quad n\to\infty.$$
Finally, because $G(n)\to\infty\text{ as } n\to\infty,$ the result follows by sending $n\to\infty,$ then $\delta\to0^+$. This concludes the proof.
\end{proof}
\end{subsection}
\begin{subsection}{Proof of Sharpness Theorem \ref{thm:sharp2}}\label{sec:sharp2} In contrast to the universal packing constant $C_p$ of Theorem \ref{thm:full} that is valid for all compact sets of positive Lebesgue measure, the packing constant $p_d$ of Theorem \ref{thm:main} is \textit{not} universal among Minkowski measurable sets of dimension $d$ with convergent first-order packing asymptotics. 
\par We prove this in three steps. First, we upper bound the packing constant for the $(1/2,1/3)$ Cantor set $T$ described in Example \ref{ex:sharp2} using the renewal theoretic approach of Lalley (Theorem 1.1 of \cite{Lalley}) and further elementary arguments. We will need specific mechanics of his proof; hence, we give an outline of the proof of Theorem 1.1 in Theorem \ref{thm:fractal}. After this, we compute $\mathcal{M}_t(T)$ using a formula for the content of an independent self-similar fractal
subset of $\mathbb{R}$ due to Falconer. Finally, we plug the formula for the Minkowski content into Main Theorem \ref{thm:main} and compare the quantities obtained in steps one and two. 
Theorem 1.1 specifically uses the deep  \textit{Continuous Renewal Theorem} from probability theory. In the upcoming paragraphs, we will thus provide an account of the relevant terminology to state the Continuous Renewal Theorem (Theorem \ref{thm:renew}) before using the result as a "black box" in the proof of Theorem \ref{thm:fractal} (as was done in Theorem 1.1.)
\begin{definition}[Arithmetic and non-arithmetic measures]
A Borel probability measure $\mu$ on $\mathbb{R}$ with support ($\text{supp}\mu$) contained in an integer lattice $h\mathbb{Z},$ for some $h>0,$ is \textit{arithmetic}. Otherwise, the measure is \textit{non-arithmetic}.
\end{definition}
In Chapter XI of \cite{Feller_1966}, Feller showed that the following strong notion of integrability is suitable for the identities of renewal theory.
\begin{definition}[Direct Riemann integrability, renewal equation]\label{def:arithrenew}
A non-negative, bounded function $z:\mathbb{R}\to\mathbb{R}$ is  \textit{directly Riemann integrable} if its upper and lower Riemann sums \textit{along the entire real line} converge to the same finite limit as the mesh of the partition defining the sums vanishes; i.e. if
$$\lim_{\delta\to0^+}\sum_{n\in\mathbb{Z}}\delta\underline{z}_n(\delta)=\lim_{\delta\to0^+}\sum_{n\in\mathbb{Z}}\delta\overline{z}_n(\delta)\in[0,\infty),$$
where 
$$\underline{z}_n(\delta)=\inf_{x\in[(n-1)\delta,n\delta)}z(x),\quad \text{ and }\quad \overline{z}_n(\delta)=\sup_{x\in[(n-1)\delta,n\delta)}z(x).$$
Direct Riemann integrability implies Riemann integrability in the ordinary limiting sense, so in this case  $$\lim_{a\to\infty}\int_{0}^{a}z(x)dx=\int_{0}^{\infty}z(x)dx\in[0,\infty).$$
Given a directly Riemann integrable $z(a),$
we say a function $Z(a)$ satisfies the \textit{renewal equation} for $z(a)$ if  
\begin{align}\label{eq:renewid}
    Z(a)=z(a)+\int_{0}^{a}Z(a-x)d\mu(x), a>0.
\end{align} 
\end{definition}
\begin{theorem}[Continuous Renewal Theorem, Chapter 11 of \cite{Feller_1966}]\label{thm:renew} Suppose $\mu$ is a non-arithmetic Borel probability measure with $\textit{supp}\mu\subset [0,\infty),$ and finite expectation
$$\overline{\mu}=\int_{0}^{\infty}xd\mu(x)<\infty.$$
If $z(x)$ is directly Riemann integrable, and $Z(a)$ satisfies Equation \ref{eq:renewid} for each $a\in[0,\infty),$ then
\begin{align*}
   \lim_{a\to\infty} Z(a)=\frac{1}{{\overline{\mu}}}\int_{0}^{\infty}z(x)dx\in[0,\infty). 
\end{align*}
\end{theorem}
\par The first author and Reznikov recently used a discrete version of this tool to prove convergence of minimal hypersingular Riesz $s$-energy along explicit subsequences of the naturals (see \cite{Anderson_Reznikov_2023} Theorem 4.1). Together with Vlasiuk and White, they proved that the continuous version yields convergence of maximal Riesz hypersingular $s$-polarization on independent self-similar fractals (\cite{ARVW2022} Theorem 15). Lalley further demonstrates the applicability of renewal theory to some geometric optimization and counting problems in \cite{Lalley_1989} and \cite{Lalley_1989b}. 
\begin{theorem}[Lalley \cite{Lalley} Theorem 1.1]\label{thm:fractal} Suppose $A\subset\mathbb{R}^p$ is an independent self-similar fractal of Minkowski dimension $d$. Then $\mathcal{N}_d(A)$ exists and is positive and finite.
\end{theorem}
\begin{proof}
Suppose $A$ has contractions $\{r_i\}_{i=1}^{M}$ and similitudes $\{\psi_i\}_{i=1}^{M}.$ Then because the leaves $\psi_1(A),\ldots, \psi_M(A),$ are pair-wise disjoint compact sets, the sets are \textit{metrically separated}; i.e. there exists $\delta>0$ such that $|x-x'|\geq\delta$ for each $x\in\psi_i(A), x'\in\psi_j(A),$ $i\neq j.$ Now, suppose $0<\varepsilon\leq\delta$ and let $N\left(\psi_i(A),\varepsilon \right)$ represent the maximal cardinality of an $\varepsilon$-separated subset of $\psi_i(A).$ If each such subset is given by $\omega_{N\left(\psi_i(A),\varepsilon \right)},$ then by the similarity between $A$ and $\psi_i(A),$ it follows that 
$$N(A,\varepsilon r_i^{-1})=N(\psi_i(A),\varepsilon).$$ Thus
$$\omega_{P(\varepsilon)}:=\bigcup_{i=1}^{M}\omega_{N\left(\psi_i(A),\varepsilon \right)}$$
is an $\varepsilon$-separated subset of $A$ of cardinality $P(\varepsilon)=\sum_{i=1}^{M}N\left(\psi_i(A),\varepsilon \right)=\sum_{i=1}^{M}N(A,\varepsilon r_i^{-1}).$ 
Hence $N(A,\varepsilon)\geq \sum_{i=1}^{M}N(A,\varepsilon r_i^{-1}).$ On the other hand, if $\omega_{Q(\varepsilon)}\subset A$ is any configuration of size greater than $\sum_{i=1}^{M}N\left(A,\varepsilon r_i^{-1}\right),$ then by the pigeonhole principle there exists $k\in\{1,\ldots,M\}$ such that $\#\left(\omega_{Q(\varepsilon)}\cap \psi_k(A)\right)\geq N(A,\varepsilon r_k^{-1})+1.$ Thus, $\delta(\omega_{Q(\varepsilon)})<\varepsilon.$ Therefore, 
\begin{align}\label{eq:keyid} N(A,\varepsilon)=\sum_{i=1}^{M}N(A,\varepsilon r_i^{-1}),  \text{ for all } 0<\varepsilon\leq \delta.
\end{align}
Hence,  
\begin{align*} 
N(A,\varepsilon)=\sum_{i=1}^{M}N(A,\varepsilon r_i^{-1})+L(\varepsilon),\text{ for all } \varepsilon>0,
\end{align*}
where $L(\varepsilon)$ is a piecewise continuous integer-valued function with finitely many discontinuities, vanishing for all $0<\varepsilon\leq \delta.$ Following the setup for the Continuous Renewal Theorem, we write
\begin{align*}
    Z(a):=e^{-ad}N(A,e^{-a}),\quad \text{ for } a>0.
\end{align*}
By Equation \ref{eq:dimensionssf}, 
\begin{align*}
    \mu:=\sum_{i=1}^{M}r_i^d\delta_{-\log(r_i)}.
\end{align*}
is a Borel probability measure on $[0,\infty),$ with finite expectation given by
$$\overline\mu=\sum_{i=1}^{M}r_i^d\log(r_i^{-1}).$$
Furthermore, $\mu$ is non-arithmetic because the set of contractions $\{r_i\}_{i=1}^{M}$ is independent.
\par 
Using Equation \ref{eq:keyid}, $Z(a)$ satisfies Equation \ref{eq:renewid}, where $z(a)$ is piecewise continuous, with only finitely many discontinuities in $[0,\infty).$ It follows that $z(a)$ is directly Riemann integrable, so Theorem \ref{thm:renew} implies 
\begin{align}\label{eq:comput}   \mathcal{N}_d(A)=\lim_{\varepsilon\to0^+}N(A,\varepsilon)\varepsilon^d=\lim_{a\to\infty}{Z(a)}=\frac{\int_{0}^{\infty}z(x)dx}{\sum_{i=1}^{M}r_i^d\log(r_i^{-1})}\in(0,\infty),
\end{align}
as $\mathcal{\underline{N}}_d(A)>0,$ for all self-similar fractals that meet Definition \ref{def:self-similar}, by Proposition \ref{thm:growth}.
\end{proof} 
\par The next result provides the other half of the critical inequality in Sharpness Theorem \ref{thm:sharp2}. Not only does the result establish Minkowski measurability of independent self-similar fractals in $\mathbb{R};$ it also provides an \textit{exact formula} for the Minkowski content of these sets that we exploit in Theorem \ref{thm:sharp2}.
\begin{theorem}[Falconer \cite{Falconer_1995} Proposition 4 (a)]\label{thm:fractalf} Suppose $A\subset\mathbb{R}$ is an independent self-similar fractal with contractions $\{r_i\}_{i=1}^{M},$ and gaps $b_1,\ldots,b_{m-1}.$ Let $d\in(0,1)$ be the Minkowski dimension of $A$ so that $\sum_{i=1}^{M}r_i^d=1.$ Then
\begin{align*}
    \mathcal{M}_d(A)=\frac{2^{1-d}}{d(1-d)}\left(\frac{\sum_{i=1}^{M}b_i^{d}}{\sum_{i=1}^{M}r_i^d\log(r_i^{-1})}\right)
\end{align*}
\end{theorem}
We are now well-positioned to prove Sharpness Theorem \ref{thm:sharp2}. 
\begin{proof}[Proof of Theorem \ref{thm:sharp2}] Directly applying the arguments and notation of Theorem \ref{thm:fractal}, we obtain 
$$\lim_{\varepsilon\to0^+}N(T,\varepsilon)\varepsilon^t=\frac{\int_{0}^{\infty}z(x)dx}{\left(\frac{1}{2}\right)^t\log(2)+\left(\frac{1}{3}\right)^t\log(3)},
$$
where $z(x)$ is given as in Equation \ref{eq:renewid}. A key ingredient of the proof is the following piecewise formula for $z(x),$ $x\in[0,\log5]\cup[\log6,\infty):$
\begin{align}\label{eqn:z} 
z(x)=
    \begin{dcases}
        2e^{-xt}, &0< x<\log2\\
        e^{-xt}, &\log2\leq x<\log3\\
        0, &\log3\leq x\leq \log5\\
        {0}, & x\geq \log6\\ 
    \end{dcases}
\end{align}
We shall now describe how to obtain Line \ref{eqn:z}. Firstly, given the containments
\begin{align}\label{eq:configs}
    \bigcup_{j\in\{1,2,3,4,6\}}\left\{\frac{i}{j}\text{ : }0\leq i\leq j\right\}\subset T,
\end{align}
and the inequality $N(T,\varepsilon)\leq N([0,1],\varepsilon),$ one deduces that the packing at $\varepsilon=\frac{1}{j},$ $j=1,2,3,4,6$ is best possible: i.e.
\begin{align}\label{eq:bestpos}
    N\left(T,\frac{1}{j}\right)=j+1,\quad \text{for} \quad j=1,2,3,4,6.
\end{align}
At the same time $\frac{1}{2}<\frac{3}{5}<\frac{2}{3},$ so $\frac{3}{5}\notin T,$ which implies $N(T,1/5)=5.$
Proving the containments in Line \ref{eq:configs} amounts to showing these points are endpoints of intervals remaining to $T$ in its recursive construction. We show this in Figure 2 using finitely many iterations of the construction. 
\begin{figure}[t]\label{fig:cantor}
\includegraphics[width=12cm]{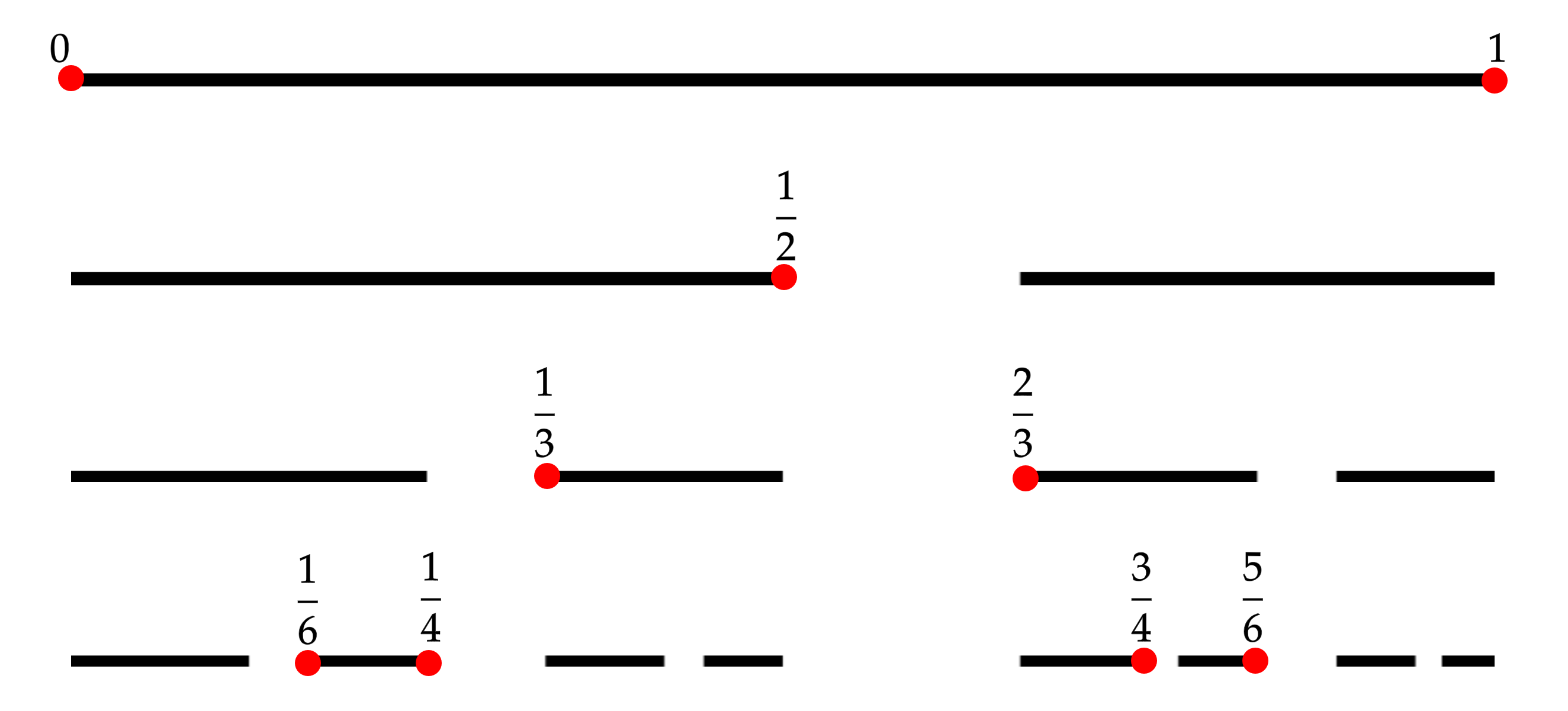}
\caption{Packing points for the $(1/2,1/3)$ Cantor set $T,$ from Line \ref{eq:bestpos}.}
\centering
\end{figure}
Using Equation \ref{eq:bestpos}, Equation \ref{eqn:z} follows by keeping track of the support of $\mu=\left({1}/{2}\right)^t\delta_{\log(2)}+\left({1}/{3}\right)^t\delta_{\log(3)},$ and by exploiting the monotonicity of $N(T,\varepsilon).$ 
Proceeding, we seek to upper bound $\mathcal{N}_t(T)$ by analyzing the behavior of $z(x)$ for $x\in(\log5,\log6).$ Using the piecewise formula for $z(x)$ we obtain
\begin{align*}
\int_0^{\infty}z(x)dx
&=\int_{0}^{\log2}2e^{-xt}dx+\int_{\log2}^{\log3}e^{-xt}dx+\int_{\log5}^{\log6}z(x)dx\\
&=\frac{2}{t}\left(1-\left(\frac{1}{2}\right)^t\right)+\frac{1}{t}\left(\left(\frac{1}{2}\right)^t-\left(\frac{1}{3}\right)^t\right)+\int_{\log5}^{\log6}z(x)dx\\
&=\frac{1}{t}\left(1-\left(\frac{1}{2}\right)^t-\left(\frac{1}{3}\right)^t+1\right)+\int_{\log5}^{\log6}z(x)dx\\
&=\frac{1}{t}+\int_{\log5}^{\log6}z(x)dx.
\end{align*}
Therefore, it remains to control the integral $\int_{\log5}^{\log6}z(x)dx.$ We consider two distinct possibilities, below.
\begin{enumerate}[{Case} 1:]
\item If $N(T,\varepsilon)=5,$ for all $\varepsilon\in(1/6,1/5),$ then $z(x)=0$ for all $x\in[\log5,\log6)$ by the arguments following Line \ref{eq:bestpos}. Hence, by Theorem \ref{thm:fractal},
$$\mathcal{N}_t(T)=\frac{1}{t\left(\left(\frac{1}{2}\right)^t\log(2)+\left(\frac{1}{3}\right)^t\log(3)\right)}<1.47.$$
On the other hand, Equation \ref{eq:comput} of Theorem 
\ref{thm:fractalf} implies that $T$ is Minkowski measurable with content given by
$$\mathcal{M}_t(T)=\frac{2^{1-t}}{t(1-t)}\frac{\left(\frac{1}{6}\right)^t}{\left(\left(\frac{1}{2}\right)^t\log(2)+\left(\frac{1}{3}\right)^t\log(3)\right)}.$$
Thus, 
\begin{align*} p_t\mathcal{M}_t(T)=&\frac{\left(\frac{1}{6}\right)^t}{t\left(\left(\frac{1}{2}\right)^t\log(2)+\left(\frac{1}{3}\right)^t\log(3)\right)}\sum_{n=1}^{\infty}\frac{n^t-(n-1)^t}{n}>1.56,
\end{align*}
where we obtain the second inequality by numerically estimating the series using finitely many terms.
Therefore, $\mathcal{N}_t(T)<p_t\mathcal{M}_t(T),$ and the desired conclusion holds.
\\
\item Otherwise, there exists a minimal $x_0\in(\log5,\log6),$ such that $N(T,e^{-x})\geq 6,$ for all $x\in[x_0,\log6).$ In this case, $z(x)=0$ for all $x\in(\log5,x_0),$ and $z(x)=e^{-xt}$ for all $x\in[x_0,\log6),$ again by the strategy following Line \ref{eq:bestpos}. Therefore,
$$\int_{\log5}^{\log6}z(x)dx=\int_{x_0}^{\log6}e^{-xt}dx=\frac{1}{t}\left(e^{-x_0t}-\left(\frac{1}{6}\right)^t\right)<\frac{1}{t}\left(\left(\frac{1}{5}\right)^t-\left(\frac{1}{6}\right)^t\right).$$
In this case, 
$$\mathcal{N}_t(T)<\frac{1+\left(\frac{1}{5}\right)^t-\left(\frac{1}{6}\right)^t}{t\left(\left(\frac{1}{2}\right)^t\log(2)+\left(\frac{1}{3}\right)^t\log(3)\right)}<1.53,$$
while again $p_t\mathcal{M}_t(T)>1.56,$
and the result holds.
\end{enumerate}
Because the result $\mathcal{N}_t(T)<p_t\mathcal{M}_t(T)$ holds in either case above, the inequality is true in general. This concludes the proof of the theorem.
\end{proof}
\begin{remark} The mechanics of the above proof reveal that the packing constant $\mathcal{N}_t(T)$ can be \textit{explicitly determined} by the value of $\delta(T,6).$ Computing this, however, appears to be no easy task.
\end{remark}
\end{subsection}
\section{Conjectures}\label{sec:CQfD}
We conjecture that Minkowski measurability is necessary and sufficient for the existence of first-order packing asymptotics on compact subsets of $\mathbb{R}.$
\begin{conjecture}\label{conj:necandsuff} Fix $d\in(0,1],$ and suppose $A\subset\mathbb{R}$ is an infinite compact set of Minkowski dimension $d.$ Then the following are equivalent:
\begin{enumerate}
  \item\label{item0} $A$ is Minkowski measurable of dimension $d$ (i.e. $\mathcal{M}_d(A)$ exists and is positive and finite.)
  \item\label{item1} The packing constant  $\mathcal{N}_d(A)$ exists and is positive and finite.
\end{enumerate}
\end{conjecture}
Theorem \ref{thm:full} establishes (1) implies (2) when $d=1.$ For $d\in(0,1),$ the Main Theorem of this manuscript (Theorem \ref{thm:main}) shows that (1) implies (2) if the complementary intervals of $A=I\setminus \cup_{j=1}^{\infty}I_j$ satisfy Property II (Monotonicity). Moreover, Sharpness Theorem \ref{thm:sharp1} demonstrates that measurability is necessary for the convergence of packing on some of these sets, without proving (2) implies (1) on monotone sets. Conditions (1) and (2) are equivalent if $A$ is a self-similar fractal in $\mathbb{R}.$ 
\par Sharpness Theorem \ref{thm:sharp2} raises many potentially difficult questions. We suspect that the lower bound provided in Line \ref{eq:p2} of Proposition \ref{thm:inv} is sub-optimal for Minkowski measurable $\mathcal{L}=(l_j)_{j=1}^{\infty}$. In light of the evidence for Conjecture \ref{conj:necandsuff}, it is therefore feasible that the quantity $N(\mathcal{L},\varepsilon)\varepsilon^d$ converges as $\varepsilon\to 0^+$ to a quantity larger than the lower bound of Line \ref{eq:p2}. We state the following conjecture.
\begin{conjecture}\label{conj:inv}
    Fix $d\in(0,1),$ and suppose that $\mathcal{L}=(l_j)_{j=1}^{\infty}$ has Minkowski dimension $d.$ Then the following are equivalent:
\begin{enumerate}
  \item\label{item00} $\mathcal{L}$ is Minkowski measurable of dimension $d$.
  \item\label{item2} The following limit exists as a finite, positive number: $$0<\lim_{\varepsilon\to0^{+}}N(\mathcal{L},\varepsilon)\varepsilon^d<\infty.$$
\end{enumerate}
\end{conjecture}
\begin{center}
    \section*{\normalsize{Acknowledgements}}
\end{center}
 The authors would like to thank several colleagues for their helpful comments during the preparation of this manuscript. Firstly, we thank Alexander Reznikov, Jonathan Schillinger, and Edward White for their participation in the Florida State University Analysis Seminar, where these problems were discussed at length. We also thank Ryan Matzke for pointing out a superfluous hypothesis in an early draft. Finally, the first author thanks the organizers of the 2023 Vanderbilt Conference; \textit{Approximation Theory and Beyond}, as well as the Fall 2023 Southeastern Sectional Meeting of the American Mathematical Society, where he presented partial progress on this project while this manuscript was still being written.
\bibliographystyle{acm}
\bibliography{main}

\begin{thebibliography}{10}

\bibitem{Anderson_Reznikov_2023}
{\sc Anderson, A., and Reznikov, A.}
\newblock {Minimal Riesz energy on balanced fractal sets}.
\newblock {\em Journal of Mathematical Analysis and Applications\/} (2023), 127663.

\bibitem{ARVW2022}
{\sc Anderson, A., Reznikov, A., Vlasiuk, O., and White, E.}
\newblock {Polarization and covering on sets of low smoothness}.
\newblock {\em Advances in Mathematics 410\/} (2022), 108720.

\bibitem{Besicovitch_Taylor_1954}
{\sc Besicovitch, A.~S., and Taylor, S.~J.}
\newblock {On the complementary intervals of a linear closed set of zero Lebesgue measure}.
\newblock {\em Journal of the London Mathematical Society s1-29}, 4 (1954), 449–459.

\bibitem{Bilyk_Mastrianni_Matzke_Steinerberger_2023}
{\sc Bilyk, D., Mastrianni, M., Matzke, R.~W., and Steinerberger, S.}
\newblock {Polarization and greedy energy on the sphere}, Jul 2023.

\bibitem{Borel}
{\sc B\'orel, E.}
\newblock {\'El\'ements de la Th\'eorie des Ensembles}.
\newblock {\em \'Editions Albin Michel\/} (1949).

\bibitem{Boro2012}
{\sc Borodachov, S.}
\newblock {Asymptotics for the minimum Riesz energy and best-packing on sets of finite packing premeasure}.
\newblock {\em Publicacions Matemàtiques 56\/} (Jan 2012), 225–254.

\bibitem{Borodachov_Hardin_Saff_2007}
{\sc Borodachov, S.~V., Hardin, D.~P., and Saff, E.~B.}
\newblock {Asymptotics of best-packing on rectifiable sets}.
\newblock {\em Proceedings of the American Mathematical Society 135}, 08 (2007), 2369–2381.

\bibitem{BHS2019}
{\sc Borodachov, S.~V., Hardin, D.~P., and Saff, E.~B.}
\newblock {Discrete energy on rectifiable sets}.
\newblock {\em Springer Monographs in Mathematics\/} (2019).

\bibitem{p-cant}
{\sc C.~Cabrelli, U.~Molter, V.~P., and Shonkwiler, R.}
\newblock {Hausdorff measure of $p$-Cantor sets.}
\newblock {\em Real Analysis Exchange 30}, 2 (2004), 413 -- 434.

\bibitem{Cabrelli_Mendivil_Molter_Shonkwiler_2004}
{\sc Cabrelli, C., Mendivil, F., Molter, U., and Shonkwiler, R.}
\newblock {On the Hausdorff $h$-measure of Cantor sets}.
\newblock {\em Pacific Journal of Mathematics 217}, 1 (Nov 2004), 45–59.

\bibitem{Caprara_Kellerer_Pferschy_2000}
{\sc Caprara, A., Kellerer, H., and Pferschy, U.}
\newblock {The multiple subset sum problem}.
\newblock {\em SIAM Journal on Optimization 11}, 2 (Jan 2000), 308–319.

\bibitem{Cohn_Elkies_2003}
{\sc Cohn, H., and Elkies, N.}
\newblock {New upper bounds on sphere packings I}.
\newblock {\em Annals of Mathematics 157}, 2 (2003), 689–714.

\bibitem{Cohn_Kumar_Miller_Radchenko_Viazovska_2017}
{\sc Cohn, H., Kumar, A., Miller, S., Radchenko, D., and Viazovska, M.}
\newblock {The sphere packing problem in dimension $24$}.
\newblock {\em Annals of Mathematics 185}, 3 (2017).

\bibitem{Falconer_1995}
{\sc Falconer, K.~J.}
\newblock {On the Minkowski measurability of fractals}.
\newblock {\em Proceedings of the American Mathematical Society 123}, 4 (1995), 1115–1115.

\bibitem{Fed96}
{\sc Federer, H.}
\newblock {\em {Geometric measure theory}}.
\newblock Springer, 1996.

\bibitem{Feller_1966}
{\sc Feller, W.}
\newblock {\em {An introduction to probability theory and its applications. Volume II}}.
\newblock Wiley, 1966.

\bibitem{Martinez}
{\sc Fernández-Martínez, M., Guirao, J., and Rodríguez-Bermúdez, G.}
\newblock More than seventy years from a milestone in fractal geometry: Moran’s theorem.
\newblock {\em Chaos: An Interdisciplinary Journal of Nonlinear Science 29\/} (01 2019), 013106.

\bibitem{Gatzouras_1999}
{\sc Gatzouras, D.}
\newblock {Lacunarity of self-similar and stochastically self-similar sets}.
\newblock {\em Transactions of the American Mathematical Society 352}, 5 (Dec 1999), 1953–1983.

\bibitem{Hales_2005}
{\sc Hales, T.}
\newblock {A proof of the Kepler conjecture}.
\newblock {\em Annals of Mathematics 162}, 3 (Nov 2005), 1065–1185.

\bibitem{HARDIN2005}
{\sc Hardin, D., and Saff, E.}
\newblock Minimal riesz energy point configurations for rectifiable d-dimensional manifolds.
\newblock {\em Advances in Mathematics 193}, 1 (2005), 174--204.

\bibitem{Hardy_Riesz_1915}
{\sc Hardy, G.~H., and Riesz, M.}
\newblock {\em {The general theory of Dirichlet’s series}}.
\newblock 1915.

\bibitem{hutch}
{\sc Hutchinson, J.~E.}
\newblock {Fractals and Self Similarity}.
\newblock {\em Indiana University Mathematics Journal 30}, 5 (1981), 713--747.

\bibitem{Hare_Mendivil_Zuberman_2013b}
{\sc K.~E.~Hare, F.~Mendivil, L.~Z.}
\newblock {The sizes of rearrangements of Cantor sets}.
\newblock {\em Canadian Mathematical Bulletin 56}, 2 (Jun 2013), 354–365.

\bibitem{kolmogorov1959}
{\sc Kolmogorov, A.~N., and Tikhomirov, V.~M.}
\newblock {$\varepsilon$-entropy and $\varepsilon$-capacity of sets in function spaces}.
\newblock {\em Uspekhi Matematicheskikh Nauk 14}, 2 (1959), 3--86.

\bibitem{Kombrink_Winter_2018b}
{\sc Kombrink, S., and Winter, S.}
\newblock {Lattice self-similar sets on the real line are not Minkowski measurable}.
\newblock {\em Ergodic Theory and Dynamical Systems 40}, 1 (Apr 2018), 221–232.

\bibitem{Lalley}
{\sc Lalley, S.~P.}
\newblock {The packing and covering functions of some self-similar fractals}.
\newblock {\em Indiana University Mathematics Journal 37}, 3 (1988), 699--709.

\bibitem{Lalley_1989b}
{\sc Lalley, S.~P.}
\newblock {Lecture notes on probabilistic methods in certain counting problems of ergodic theory}, May 1989.

\bibitem{Lalley_1989}
{\sc Lalley, S.~P.}
\newblock {Renewal theorems in symbolic dynamics, with applications to geodesic flows, noneuclidean tessellations and their fractal limits}.
\newblock {\em Acta Mathematica 163}, 0 (1989), 1–55.

\bibitem{Lapidus_Maier_1995b}
{\sc Lapidus, M.~L., and Maier, H.}
\newblock {The Riemann hypothesis and inverse spectral problems for fractal strings}.
\newblock {\em Journal of the London Mathematical Society 52}, 1 (Aug 1995), 15–34.

\bibitem{Lapidus1993}
{\sc Lapidus, M.~L., and Pomerance, C.}
\newblock {The Riemann Zeta-Function and the One-Dimensional Weyl-Berry Conjecture for Fractal Drums}.
\newblock {\em Proceedings of The London Mathematical Society 66\/} (1993), 41--69.

\bibitem{López-García_McCleary_2022}
{\sc López-García, A., and McCleary, R.~E.}
\newblock {Asymptotics of the minimum values of Riesz and logarithmic potentials generated by greedy energy sequences on the unit circle}.
\newblock {\em Journal of Mathematical Analysis and Applications 508}, 1 (Apr 2022), 125866.

\bibitem{Mattila_2004}
{\sc Mattila, P.}
\newblock {\em {Geometry of sets and measures in euclidean spaces: Fractals and Rectifiability}}.
\newblock Cambridge University Press, 2004.

\bibitem{Moran_1946}
{\sc Moran, P. A.~P.}
\newblock Additive functions of intervals and hausdorff measure.
\newblock {\em Mathematical Proceedings of the Cambridge Philosophical Society 42}, 1 (1946), 15–23.

\bibitem{Musin2015}
{\sc Musin, O.~R., and Tarasov, A.~S.}
\newblock {The Tammes problem for $N = 14$}.
\newblock {\em Experimental Mathematics 24}, 4 (2015), 460–468.

\bibitem{Rolfes}
{\sc Rolfes, J.~H.}
\newblock {\em {Convex optimization techniques for geometric covering problems}}.
\newblock PhD thesis.

\bibitem{PML1930}
{\sc Tammes, P.}
\newblock {\em {On the origin of number and arrangement of the places of exit on the surface of pollen-grains}}.
\newblock J. H. de Bussy, 1930.

\bibitem{Fejes_1942}
{\sc T\"oth, L.~F.}
\newblock {Ber Die Dichteste Kugellagerung}.
\newblock {\em Mathematische Zeitschrift 48}, 1 (Dec 1942), 676–684.

\bibitem{Viazovska_2017}
{\sc Viazovska, M.}
\newblock {The sphere packing problem in dimension $8$}.
\newblock {\em Annals of Mathematics 185}, 3 (2017).

\bibitem{Whittaker_Watson_1920b}
{\sc Whittaker, E.~T., and Watson, G.~N.}
\newblock {\em {A course of modern analysis: An introduction to the general theory of infinite processes and of analytical functions; With an account of the principal transcendental functions}}.
\newblock Cambridge University Press, 1920.

\bibitem{Xiong_Wu_2009}
{\sc Y.~Xiong, M.~W.}
\newblock {Category and dimensions for cut-out sets}.
\newblock {\em Journal of Mathematical Analysis and Applications 358}, 1 (Oct 2009), 125–135.

\bibitem{yudin}
{\sc Yudin, D.~B., and Golshtein, E.~G.}
\newblock {\em {Lineino Programmirovanie}}.
\newblock Gos. Izd-vo Fisiko-Matematicheskoii Lit-ry, 1963.

\end{thebibliography}
\end{document}